\documentclass[reqno,tbtags]{amsproc}
\usepackage{amssymb}
\usepackage{epsfig}
\usepackage{txfonts}
\numberwithin{equation}{section} 
\newtheorem{theorem}{Theorem}[section]
\newtheorem{corollary}{Corollary}[section]
\newtheorem{lemma}{Lemma}[section]
\newtheorem{assertion}{Assertion}[section]
\theoremstyle{definition}
\newtheorem{remark}{Remark}
\def\paramT{\delta}

\def\paramY{\rho}

\def\eqDef{:={}}
\renewcommand{\P}{\mathsf{P}}
\newcommand{\D}{\mathsf{D}\hskip 1pt}
\newcommand{\E}{\mathsf{E}\hskip 1pt}
\newcommand{\p}[1]{\mathsf{p}_{#1}}
\newcommand{\Rline}{\mathsf{R}}
\newcommand{\BesselI}[1]{I_{#1}}
\newcommand{\BesselK}[1]{K_{#1}}
\newcommand{\DerC}[2]{\mathsf{C}_{\,#1}^{[#2]}}
\newcommand{\DerU}[2]{\mathsf{U}_{\,#1}^{[#2]}}
\newcommand{\Ruat}{u_{\alpha,t}}
\newcommand{\Huat}{u_{\alpha,t}^{[\mathcal{M}]}}
\newcommand{\UGauss}[2]{\varPhi_{\left({#1},{#2}\right)}}
\newcommand{\Ugauss}[2]{\varphi_{\left({#1},{#2}\right)}}
\newcommand{\adjustL}{\varkappa}
\newcommand{\homN}[1]{N_{#1}}
\newcommand{\homV}[1]{V_{#1}}
\newcommand{\homR}[1]{R_{#1}}
\newcommand{\homM}[1]{M_{#1}}
\newcommand{\Tich}[1]{\mathit{\Upsilon}_{#1}}
\newcommand{\Elem}[2]{\mathcal{I}^{[#2]}_{#1}}
\newcommand{\AInt}[2]{{\mathcal{#1}}_{#2}}
\newcommand{\EnOne}{\epsilon\,n_{u+cv}}
\newcommand{\T}[1]{T_{#1}}
\newcommand{\Y}[1]{Y_{#1}}
\newcommand{\muIG}{\mu}
\newcommand{\lambdaIG}{\lambda}
\newcommand{\HmuIG}{\hat{\mu}}
\newcommand{\NQuant}[1]{\kappa_{#1}}
\newcommand{\funU}[1]{\mathsf{z}_{#1}}
\newcommand{\parA}[1]{a_{#1}}
\newcommand{\parB}[1]{b_{#1}}
\newcommand{\cS}{c^{\ast}}

\def\ASB{\emph{ASTIN Bulletin\/}}

\def\AMS{\emph{Ann. Math. Statist.\/}}
\def\BM{\emph{Biometrica\/}}
\def\DAN{\emph{Doklady Akademii Nauk\/}}
\def\IME{\emph{In\-su\-ran\-ce: Ma\-the\-ma\-tics and Eco\-no\-mics\/}}

\def\JRSS{\emph{Jour\-nal of the Royal Sta\-tist. Soc., Ser. B\/}}
\def\JMAA{\emph{J. Math. Anal. Appl.\/}}

\def\JHE{\emph{Journal of Hydrologic Engineering\/}}

\def\RSA{\emph{Revue de Statistique Appliqu{\'e}e\/}}
\def\Risks{\emph{Risks\/}}
\def\RE{\emph{Radiotech. Electron.\/}}

\def\TPA{\emph{Theory Probab. Appl.\/}}

\begin{document}

\author[Vsevolod K. Malinovskii]{Vsevolod K. Malinovskii\footnote{This work was
supported by RFBR (grant No.~19-01-00045).}}


\address{Central Economics and Mathematics Institute (CEMI) of Russian Academy of Science,
117418, Nakhimovskiy prosp., 47, Moscow, Russia}

\email{Vsevolod.Malinovskii@mail.ru, admin@actlab.ru}

\urladdr{http:/\!/www.actlab.ru}

\title[APPROXIMATION FOR THE FIXED-PROBABILITY LEVEL]{APPROXIMATION OF THE FIXED-PROBABILITY
LEVEL\\[0pt]FOR A COMPOUND RENEWAL PROCESS}

\maketitle

\begin{abstract}
Dealing with compound renewal process with generally distributed jump sizes and
inter-renewal intervals, we focus on the approximation for the
fixed-probability level, which is the core of inverse level crossing problem.
We are developing an analytical technique presented in
\cite{[Malinovskii=2017a]}--\cite{[Malinovskii=2018]} and based on Kendall's
identity; this yields (see \cite{[Malinovskii=2018=dan=1]}) inverse Gaussian
approximation in the direct level crossing problem. These issues are of great
importance in risk theory.
\end{abstract}

\section{Introduction}\label{dgfjhngjk}

In this paper, we will be focussed on $\T{1}$, $\T{i}\overset{d}{=}T$,
$i=2,3,\dots$, with p.d.f. $f_{\T{1}}$ and $f_{T}$, which are independent
positive random variables, called intervals between renewals, and on
$\Y{i}\overset{d}{=}Y$, $i=1,2,\dots$, with p.d.f. $f_{Y}$, which are
independent positive random variables called jump sizes at the moments of
renewals. We assume that these sequences are mutually independent and
$\T{1}\overset{d}{=}T$, i.e., we confine ourselves to the ordinary renewal
process $\homN{s}\eqDef\max\big\{n>0:\sum_{i=1}^{n}\T{i}\leqslant s\big\}$,
with $\homN{s}\eqDef 0$, if $\T{1}>s$. Moreover, we focus on p.d.f.
$f_{\T{1}}$, $f_{T}$, and $f_{Y}$ bounded above by a finite constant; these
restriction may be relaxed, but not in this paper.

Let us introduce the random process
\begin{equation}\label{sdghtyjrtYY}
\homR{s}\eqDef u+cs-\homV{s},\quad s\geqslant 0,
\end{equation}
where $u\geqslant 0$ and $c\geqslant 0$ are constants,
$\homV{s}\eqDef\sum_{i=1}^{\homN{s}}\Y{i}$, with $\homV{s}\eqDef 0$, if
$\T{1}>s$. The random process $\homV{s}$, $s\geqslant 0$, is called compound
(ordinary) renewal processes; its trajectories are piecewise linear.

By the \emph{direct} level crossing problem we call the study of probability
$\P\big\{\Tich{u,c}\leqslant t\big\}$, where
\begin{equation}\label{wq4t5g4ywe}
\begin{aligned}
\Tich{u,c}&\eqDef\inf\left\{s>0:\homV{s}-cs>u\right\}
\\
&=\inf\left\{s>0:\homR{s}<0\right\},
\end{aligned}
\end{equation}
or $+\infty$, if $\homV{s}-cs\leqslant u$ for all $s>0$, whereas the
\emph{inverse} level crossing problem is focussed on the study of a solution
(with respect to $u$) to the equation
\begin{equation}\label{rtyujtkjtyk}
\P\big\{\Tich{u,c}\leqslant t\big\}=\alpha,
\end{equation}
where $\alpha$ is positive and reasonably small, e.g., $\alpha=0{.}05$. This
solution is denoted by $\Ruat(c)$, $c\geqslant 0$, and is called
fixed-probability level. It is easily seen that $\P\{\Tich{u,c}\leqslant
t\}=\P\big\{\inf_{0\leqslant s\leqslant t}\homR{s}<0\big\}$, and the left-hand
side of \eqref{rtyujtkjtyk} can be rewritten accordingly.

The fixed-probability level defined by equation \eqref{rtyujtkjtyk} is an
implicit function. Its analysis is based on a detailed study of the probability
in the left-hand side of \eqref{rtyujtkjtyk}, i.e., on the direct level
crossing problem. When the random variables $T$ and $Y$ are exponentially
distributed with parameters $\paramT>0$ and $\paramY>0$, the random process
$\homN{s}$, $s\geqslant 0$, is a Poisson process with intensity $\paramT$. In
this case, $\P\big\{\homV{s}\leqslant x\big\}$, $\E\big(\homV{s}\big)$,
$\D\big(\homV{s}\big)$, and $\P\big\{\Tich{u,c}\leqslant t\big\}$ are expressed
in a closed form, using elementary and special functions such as the modified
Bessel functions $\BesselI{n}(z)$ of the first kind of order $n$. Consequently,
equation \eqref{rtyujtkjtyk} is written explicitly and the study of an implicit
function $\Ruat(c)$, $c\geqslant 0$, is carried out in
\cite{[Malinovskii=2012]}, \cite{[Malinovskii=2014b]}; it goes along the road
map set before in \cite{[Malinovskii=2009]}, \cite{[Malinovskii=2014a]} in the
diffusion model.

In the case of generally distributed random variables $T$ and $Y$, to find a
solution to the direct (let alone inverse) level crossing problem in terms of
elementary and special functions seems impossible, except for a few very
special cases, whence our attention to approximations of
$\P\big\{\Tich{u,c}\leqslant t\big\}$, as $u\to\infty$ and $u,t\to\infty$, and
of $\Ruat(c)$, $c\geqslant 0$, as $t\to\infty$.

We proceed investigating $\Ruat(c)$, $c\geqslant 0$, from the inverse Gaussian
approximation\footnote{This stands out from a number of previously known
approximations, of which Cram{\'e}r's and diffusion, obtained by means of the
invariance principle, are the most famous.} obtained in
\cite{[Malinovskii=2018=dan=1]},
\cite{[Malinovskii=2017a]}--\cite{[Malinovskii=2018]}. In a nutshell, we use
Kendall's identity for $\P\big\{\Tich{u,c}\leqslant t\big\}$, which expresses
this probability through convolution powers of $f_{T}$ and $f_{Y}$; the central
limit theory is then applied to them. This method is widely applicable, e.g.,
it allows us to find approximations for the first-order derivatives
$\frac{\partial}{\partial c}\,\P\big\{\Tich{u,c}\leqslant t\big\}$ and
$\frac{\partial}{\partial u}\,\P\big\{\Tich{u,c}\leqslant t\big\}$, and even
for higher-order derivatives, such as $\frac{\partial^2}{\partial
c^2}\,\P\big\{\Tich{u,c}\leqslant t\big\}$ and $\frac{\partial^2}{\partial
u^2}\,\P\big\{\Tich{u,c}\leqslant t\big\}$. This issue is central (see
Theorem~\ref{etrtjt}) in the study of, e.g., monotony and convexity of the
implicitly defined function $\Ruat(c)$, $c\geqslant 0$.

This approach is aimed at obtaining a large set of results using standard
techniques. Such results include approximations and estimates of the rate of
convergence, as well as various refinements, e.g., asymptotic expansions. The
main focus is on the diversity and accuracy of the results, rather than the
minimality of technical conditions, although the conditions in these results
are close to minimal\footnote{Compare with \cite{[Borovkov=2015]}, where
minimization of conditions and a more general level are focussed.}.

In applications, both direst and inverse level crossing problems are of a great
importance. In risk theory, the function $\Ruat(c)$, $c\geqslant 0$, models a
non-ruin capital\footnote{Apparently, this mathematical concept, viewed as an
implicit function, was focussed straightforwardly for the first time in
\cite{[Malinovskii=2012]}, where it was referred to as ``level capital'', or
``$\alpha$-level initial capital'' (see \cite{[Malinovskii=2012]},
Definition~3.1). In \cite{[Malinovskii=2014b]}, the term ``ruin capital'',
emphasizing its role in matters of solvency, was used instead; if one seeks to
escape ruin, the term ``non-ruin capital'' sounds more appropriate.} that makes
the probability of ruin over time $t$ equal to a predetermined value $\alpha$,
chosen as an acceptable degree of insolvency. This academic concept is related
to fundamental methods of insurance solvency's regulation (see, e.g.,
\cite{[Beard-et-al.-1984]}, \cite{[Daykin-et-al.-1996]},
\cite{[Pentikainen-et-al.-1989]}, \cite{[Sandstrom-2006]},
\cite{[Sandstrom=2011]}); in practice, they are mainly implemented by
simulation. Analytically, as a problem of collective risk theory, the inverse
level crossing problem was first investigated in \cite{[Malinovskii=2012]} (see
also \cite{[Malinovskii=2014b]}), where equitable solvent controls in a
multi-period game model of risk were considered; as a partial single-period
model, Lunderg's model with exponentially distributed claim size was focussed
in \cite{[Malinovskii=2012]}; similar issue in the diffusion risk model was
investigated in \cite{[Malinovskii=2014a]}.

The rest of this paper is arranged as follows. In Section~\ref{srteyrjfr}, we
recall Kendall's identity. In Section~\ref{rgerhyryttryjh}, we derive similar
identities for derivatives $\frac{\partial}{\partial
c}\,\P\big\{\Tich{u,c}\leqslant t\big\}$ and $\frac{\partial}{\partial
u}\,\P\big\{\Tich{u,c}\leqslant t\big\}$. In Section~\ref{asrgrhger}, we
outline the inverse Gaussian approximation (see
\cite{[Malinovskii=2018=dan=1]}) in the direct level crossing problem; this
result is presented in detail in \cite{[Malinovskii=2017a]},
\cite{[Malinovskii=2017b]}, \cite{[Malinovskii=Malinovskii=2017]}. In
Section~\ref{wqstgerhs}, we establish approximations for derivatives, using the
same stages as in the proof of inverse Gaussian approximation. In
Section~\ref{srgterjer}, which is core of this paper, we focus on
approximations in the inverse level crossing problem: first, we deal with
structural results, then with monotony and convexity\footnote{Recall that a
differentiable function is convex (i.e., has the form $\smile$) if its second
derivative is positive.} results, and finally with heuristic fixed-probability
level and with elementary bounds on the fixed-probability level, which are a
tool for numerical calculations.

\section{Kendall's identity: a keystone result}\label{srteyrjfr}

Let us introduce\footnote{The $\inf$-definition for $\homM{x}$, $x>0$, in
contrast to equivalent $\max$-definition for $\homN{s}$, $s\geqslant 0$, is a
hint on the difference between these renewal processes.}
\begin{equation}\label{erthtjthew}
\homM{x}\eqDef\inf\left\{k\geqslant 1:\sum_{i=1}^{k}Y_{i}>x\,\right\}-1,\quad
x>0,
\end{equation}
which is a renewal process generated by the random variables $\Y{i}$,
$i=1,2,\dots$. The following result is known as Kendall's identity (see first
\cite{[Kendall=1957]}, and then
\cite{[Borovkov=1965]}--\cite{[Borovkov=Dickson=2008]}, \cite{[Keilson=1963]},
\cite{[Rogozin=1966]}, \cite{[Skorohod=1991]}, \cite{[Zolotarev=1964]}).

\begin{assertion}[Kendall's identity]\label{23456uy7}
With $0<v<t$, we have
\begin{equation}\label{dfgbehredf}
\begin{aligned}
\P\big\{v<\Tich{u,c}\leqslant
t\mid\T{1}=v\big\}&=\int_{v}^{t}\dfrac{u+cv}{u+cz}\;
\p{\,\sum_{i=2}^{\homM{u+cz}+1}\T{i}}(z-v)\,dz
\\[0pt]
&=\int_{v}^{t}\dfrac{u+cv}{u+cz}\sum_{n=1}^{\infty}
\P\big\{\homM{u+cz}=n\big\}\,f_{T}^{*n}(z-v)\,dz.
\end{aligned}
\end{equation}
\end{assertion}

Proceeding from Assertion~\ref{23456uy7} and using the equality
\begin{equation}\label{ewrktulertye}
\begin{aligned}
\P\big\{\Tich{u,c}\leqslant
t\big\}&=\int_{0}^{t}\P\big\{u+cv-Y<0\big\}\,f_{\T{1}}(v)\,dv
\\[0pt]
&+\int_{0}^{t}\P\big\{v<\Tich{u,c}\leqslant
t\mid\T{1}=v\big\}\,f_{\T{1}}(v)\,dv,
\end{aligned}
\end{equation}
we switch back to (unconditional) distribution of the level crossing time.

The identity \eqref{dfgbehredf} may be rewritten exclusively in terms of
$n$-fold convolutions $f_{T}^{*n}$ and $f_{Y}^{*n}$. Indeed, bearing in mind
that $\Y{i}$, $i=1,2,\dots$, are i.i.d., we have
\begin{equation*}
\begin{aligned}
\P\big\{\homM{u+cv+cy}=n\big\}&=\P\left\{\sum_{i=1}^{n}\Y{i}\leqslant
u+cv+cy<\sum_{i=1}^{n+1}\Y{i}\right\}
\\[0pt]
&=\int_{0}^{u+cv+cy}f^{*n}_{Y}(u+cv+cy-z)\,\P\{\Y{n+1}>z\}\,dz.
\end{aligned}
\end{equation*}
Making the change of variables $y=z-v$ in \eqref{dfgbehredf}, we rewrite it as
\begin{equation}\label{wertyrjhn}
\begin{aligned}
\P\big\{v<\Tich{u,c}\leqslant t\mid\T{1}=v\big\}
&=\sum_{n=1}^{\infty}\int_{0}^{t-v}\frac{u+cv}{u+cv+cy}
\int_{0}^{u+cv+cy}\P\big\{\Y{n+1}>z\big\}
\\[2pt]
&\times f_{Y}^{*n}(u+cv+cy-z)\,f_{T}^{*n}(y)\,dy\,dz.
\end{aligned}
\end{equation}

Equality \eqref{dfgbehredf}, or its copy \eqref{wertyrjhn}, and equality
\eqref{ewrktulertye} are fundamental in a series of approximations and
closed-form results presented in \cite{[Malinovskii=2012]},
\cite{[Malinovskii=2014b]}--\cite{[Malinovskii=Malinovskii=2017]}. In
particular, when $T$ and $Y$ are exponentially distributed with parameters
$\paramT>0$ and $\paramY>0$, closed-form expressions for
$\P\big\{\Tich{u,c}\leqslant t\big\}$ follow from the next corollary of
Assertion~\ref{23456uy7}.

\begin{corollary}\label{etyjhtrjrt}
For $Y$ exponentially distributed with parameter $\paramY>0$, we have
\begin{equation}\label{wedteyjtkj}
\begin{aligned}
\P\big\{\Tich{u,c}\leqslant
t\big\}&=\int_{0}^{t}e^{-\paramY\,(u+cs)}\,\Bigg(f_{\T{1}}(s)
+\frac{1}{u+cs}\sum_{n=1}^{\infty}\frac{\big(\paramY\,(u+cs)\big)^n}{n!}
\\[0pt]
&\times\int_{0}^{s}(u+cv)f_{T}^{*n}(s-v)f_{\T{1}}(v)\,dv\Bigg)\,ds.
\end{aligned}
\end{equation}
\end{corollary}

\begin{proof}[Proof of Corollary~\ref{etyjhtrjrt}]\label{eartfygul}
For $\Y{i}\overset{d}{=}Y$, $i=1,2,\dots$, when $Y$ is exponentially
distributed with parameter $\paramY$, we have
$\P\big\{u+cv-Y<0\big\}=e^{-\paramY\,(u+cv)}$ and
\begin{equation*}
\P\big\{\homM{u+cs}=n\big\}=e^{-\paramY\,(u+cs)}\,\frac{\big(\paramY\,(u+cs)\big)^n}{n!},\quad
n=1,2,\dots,
\end{equation*}
whence equality \eqref{ewrktulertye} rewrites as \eqref{wedteyjtkj}.
\end{proof}

\section{Derivatives of $\P\big\{\Tich{u,c}\leqslant t\big\}$ via Kendall's identity}\label{rgerhyryttryjh}

Let p.d.f. $f_{T}$ and $f_{Y}$ be differentiable. Kendall's identity
\eqref{dfgbehredf}, or its copy \eqref{wertyrjhn}, and equality
\eqref{ewrktulertye} allow us to express the derivatives of
$\P\big\{\Tich{u,c}\leqslant t\big\}$ with respect to $c$ and $u$ in a similar
way.

\subsection{Derivative $\frac{\partial}{\partial c}\,\P\big\{\Tich{u,c}\leqslant t\big\}$}\label{sdtyjfer}

Let us start with the derivative of $\P\big\{\Tich{u,c}\leqslant t\big\}$ with
respect to $c$ and introduce the following expressions:
\begin{equation}\label{sdfghfghnfgn}
\begin{aligned}
\DerC{u,c}{1}(t\mid
v)&=-\sum_{n=1}^{\infty}\int_{0}^{t-v}\frac{uy}{(u+cv+cy)^{\,2}}
\\[0pt]
&\times\int_{0}^{u+cv+cy}\P\big\{\Y{n+1}>z\big\}
f_{Y}^{*n}\big(u+cv+cy-z\big)\,dz\,f_{T}^{*n}(y)\,dy,
\\[0pt]
\DerC{u,c}{2}(t\mid
v)&=\sum_{n=1}^{\infty}\int_{0}^{t-v}\frac{(u+cv)\,(v+y)}{u+cv+cy}
\\[0pt]
&\times\int_{0}^{u+cv+cy}\P\big\{\Y{n+1}>z\big\}\,\Bigg\{\,f_{Y}(0)f_{Y}^{*(n-1)}\big(u+cv+cy-z\big)
\\[0pt]
&+\int_{0}^{u+cv+cy-z}f_{Y}^{\,\prime}(\xi)
f_{Y}^{*(n-1)}\big(u+cv+cy-z-\xi\big)\,d\xi\,\Bigg\}\,dz\,
\\[0pt]
&\times f_{T}^{*n}(y)\,dy,
\\[0pt]
\DerC{u,c}{3}(t\mid
v)&=\sum_{n=1}^{\infty}f_{Y}^{*n}(0)\int_{0}^{t-v}\frac{(u+cv)\,(v+y)}{u+cv+cy}\,
\P\big\{\Y{n+1}>u+cv+cy\big\}\,
\\[0pt]
&\times f_{T}^{*n}(y)\,dy.
\end{aligned}
\end{equation}

\begin{lemma}\label{yukiyulu}
For $c>0$, $u>0$, $t>v>0$, we have
\begin{equation}\label{wrgrhrtrr}
\begin{aligned}
\frac{\partial}{\partial c}\,\P\big\{\Tich{u,c}\leqslant
t\big\}&=-\int_{0}^{t}f_{Y}(u+cv)\,v\,f_{\T{1}}(v)\,dv
\\[0pt]
&+\int_{0}^{t}\frac{\partial}{\partial c}\,\P\big\{v<\Tich{u,c}\leqslant
t\mid\T{1}=v\big\}\,f_{\T{1}}(v)\,dv,
\end{aligned}
\end{equation}
where
\begin{equation}\label{srdtyrjrtj}
\frac{\partial}{\partial c}\,\P\big\{v<\Tich{u,c}\leqslant
t\mid\T{1}=v\big\}=\DerC{u,c}{1}(t\mid v)+\DerC{u,c}{2}(t\mid v)
+\DerC{u,c}{3}(t\mid v).
\end{equation}
\end{lemma}

\begin{proof}[Proof of Lemma~\ref{yukiyulu}]
The proof is based on identities \eqref{ewrktulertye} and \eqref{wertyrjhn},
i.e.,
\begin{equation}\label{ws4d5uhjtd}
\begin{aligned}
\P\big\{\Tich{u,c}\leqslant
t\big\}&=\int_{0}^{t}\P\big\{u+cv-Y<0\big\}\,f_{\T{1}}(v)\,dv
\\[0pt]
&+\int_{0}^{t}\P\big\{v<\Tich{u,c}\leqslant
t\mid\T{1}=v\big\}\,f_{\T{1}}(v)\,dv
\end{aligned}
\end{equation}
and
\begin{equation}\label{adsgfsfbdfnb}
\begin{aligned}
\P\big\{v<\Tich{u,c}\leqslant t\mid\T{1}=v\big\}
&=\sum_{n=1}^{\infty}\int_{0}^{t-v}\frac{u+cv}{u+cv+cy}
\int_{0}^{u+cv+cy}\P\big\{\Y{n+1}>z\big\}
\\[2pt]
&\times f_{Y}^{*n}\big(u+cv+cy-z\big)\,f_{T}^{*n}(y)\,dy\,dz.
\end{aligned}
\end{equation}

Differentiating \eqref{adsgfsfbdfnb}, we have
\begin{equation*}
\begin{aligned}
\frac{\partial}{\partial c}\,\P\big\{v<\Tich{u,c}\leqslant t\mid\T{1}=v\big\}
&=\sum_{n=1}^{\infty}\int_{0}^{t-v}\frac{\partial}{\partial
c}\,\Bigg(\frac{u+cv}{u+cv+cy} \int_{0}^{u+cv+cy}\P\big\{\Y{n+1}>z\big\}
\\[2pt]
&\times f_{Y}^{*n}(u+cv+cy-z)\,dz\,\Bigg)\,f_{T}^{*n}(y)\,dy.
\end{aligned}
\end{equation*}
The integrand is
\begin{equation*}
\begin{aligned}
\frac{\partial}{\partial c}&\,\Bigg(\frac{u+cv}{u+cv+cy}
\int_{0}^{u+cv+cy}\P\big\{\Y{n+1}>z\big\}\,f_{Y}^{*n}(u+cv+cy-z)\,dz\,\Bigg)
\\[0pt]
&=\frac{\partial}{\partial c}\,\Bigg(\frac{u+cv}{u+cv+cy}\,\Bigg)
\int_{0}^{u+cv+cy}\P\big\{\Y{n+1}>z\big\}\,f_{Y}^{*n}(u+cv+cy-z)\,dz
\\[0pt]
&+\frac{u+cv}{u+cv+cy}\,\frac{\partial}{\partial
c}\,\Bigg(\int_{0}^{u+cv+cy}\P\big\{\Y{n+1}>z\big\}\,f_{Y}^{*n}(u+cv+cy-z)\,dz\,\Bigg),
\end{aligned}
\end{equation*}
where
\begin{equation*}
\dfrac{\partial}{\partial
c}\,\Bigg(\dfrac{u+cv}{u+cv+cy}\,\Bigg)=-\dfrac{uy}{(u+cv+cy)^2}
\end{equation*}
and
\begin{equation*}
\begin{aligned}
\frac{\partial}{\partial
c}&\,\Bigg(\int_{0}^{u+cv+cy}\P\big\{\Y{n+1}>z\big\}\,f_{Y}^{*n}(u+cv+cy-z)\,dz\,\Bigg)
\\[0pt]
&=\int_{0}^{u+cv+cy}\P\big\{\Y{n+1}>z\big\}\,\Bigg(\frac{\partial}{\partial
c}\,f_{Y}^{*n}(u+cv+cy-z)\Bigg)\,dz
\\[0pt]
&+(v+y)\,\P\big\{\Y{n+1}>u+cv+cy\big\}\,f_{Y}^{*n}(0).
\end{aligned}
\end{equation*}
For $n\geqslant 2$, differentiation of $n$-fold convolutions yields
\begin{equation*}
\begin{aligned}
\frac{\partial}{\partial c}\,f_{Y}^{*n}(u+cv+cy-z)&=\frac{\partial}{\partial
c}\int_{0}^{u+cv+cy-z}f_{Y}\left(\left(u+cv+cy-z\right)-\zeta\right)f_{Y}^{*(n-1)}(\zeta)\,d\zeta
\\[0pt]
&=(v+y)\underbrace{\int_{0}^{u+cv+cy-z}f_{Y}^{\,\prime}\left(\left(u+cv+cy-z\right)-\zeta\right)
f_{Y}^{*(n-1)}(\zeta)\,d\zeta}_{\int_{0}^{u+cv+cy-z}f_{Y}^{\,\prime}(\xi)f_{Y}^{*(n-1)}(\left(u+cv+cy-z\right)-\xi)\,d\xi}
\\[0pt]
&+(v+y)\,f_{Y}(0)f_{Y}^{*(n-1)}(u+cv+cy-z),
\end{aligned}
\end{equation*}
whence, by elementary calculations, the result.
\end{proof}

\subsection{Derivative $\frac{\partial}{\partial u}\,\P\big\{\Tich{u,c}\leqslant t\big\}$}\label{asrdthdgjfg}

Let proceed with the derivative of $\P\big\{\Tich{u,c}\leqslant t\big\}$ with
respect to $u$ and introduce the following expressions:
\begin{equation*}
\begin{aligned}
\DerU{u,c}{1}(t\mid v)&=\sum_{n=1}^{\infty}\int_{0}^{t-v}
\frac{cy}{(u+cv+cy)^{\,2}}
\\[-2pt]
&\times\,\int_{0}^{u+cv+cy} \P\big\{\Y{n+1}>z\big\}\,
f_{Y}^{*n}(u+cv+cy-z)\,f_{T}^{*n}(y)\,dy\,dz,
\\[0pt]
\DerU{u,c}{2}(t\mid v)&=\sum_{n=1}^{\infty}\int_{0}^{t-v}\frac{u+cv}{u+cv+cy}
\\[-2pt]
&\times\,\int_{0}^{u+cv+cy}\P\big\{\Y{n+1}>z\big\}\,
\Bigg\{\,f_{Y}(0)f_{Y}^{*(n-1)}(u+cv+cy-z)
\\[0pt]
&+\int_{0}^{u+cv+cy-z}
f_{Y}^{\,\prime}(\xi)f_{Y}^{*(n-1)}(\left(u+cv+cy-z\right)-\xi)\,d\xi\,\Bigg\}
\\[0pt]
&\times f_{T}^{*n}(y)\,dy\,dz,
\\[0pt]
\DerU{u,c}{3}(t\mid v)&=\sum_{n=1}^{\infty}f_{Y}^{*n}(0)\int_{0}^{t-v}
\frac{u+cv}{u+cv+cy}\,\P\big\{\Y{n+1}>u+cv+cy\big\}\,
\\[0pt]
&\times f_{T}^{*n}(y)\,dy.
\end{aligned}
\end{equation*}

\begin{lemma}\label{ytuktyktly}
For $c>0$, $u>0$, $t>v>0$, we have
\begin{equation}\label{srdtyutryitk}
\begin{aligned}
\frac{\partial}{\partial u}\,\P\big\{\Tich{u,c}\leqslant
t\big\}&=-\int_{0}^{t}f_{Y}(u+cv)\,v\,f_{\T{1}}(v)\,dv
\\[0pt]
&+\int_{0}^{t}\frac{\partial}{\partial u}\,\P\big\{v<\Tich{u,c}\leqslant
t\mid\T{1}=v\big\}\,f_{\T{1}}(v)\,dv,
\end{aligned}
\end{equation}
where
\begin{equation}\label{serdtgfjkyfg}
\frac{\partial}{\partial u}\,\P\big\{v<\Tich{u,c}\leqslant
t\mid\T{1}=v\big\}=\DerU{u,c}{1}(t\mid v)+\DerU{u,c}{2}(t\mid v)
+\DerU{u,c}{3}(t\mid v).
\end{equation}
\end{lemma}

\begin{proof}[Proof of Lemma~\ref{ytuktyktly}]
Differentiating identity \eqref{adsgfsfbdfnb}, we have
\begin{equation*}
\begin{aligned}
\frac{\partial}{\partial u}\,\P\big\{v<\Tich{u,c}\leqslant t\mid\T{1}=v\big\}
&=\sum_{n=1}^{\infty}\int_{0}^{t-v}\frac{\partial}{\partial
u}\,\Bigg(\frac{u+cv}{u+cv+cy} \int_{0}^{u+cv+cy}\P\big\{\Y{n+1}>z\big\}
\\[0pt]
&\times f_{Y}^{*n}(u+cv+cy-z)\,dz\,\Bigg)\,f_{T}^{*n}(y)\,dy.
\end{aligned}
\end{equation*}
The integrand is
\begin{equation*}
\begin{aligned}
\frac{\partial}{\partial u}&\,\Bigg(\frac{u+cv}{u+cv+cy}
\int_{0}^{u+cv+cy}\P\big\{\Y{n+1}>z\big\}\,f_{Y}^{*n}(u+cv+cy-z)\,dz\,\Bigg)
\\[0pt]
&=\frac{\partial}{\partial u}\,\Bigg(\frac{u+cv}{u+cv+cy}\,\Bigg)
\int_{0}^{u+cv+cy}\P\big\{\Y{n+1}>z\big\}\,f_{Y}^{*n}(u+cv+cy-z)\,dz
\\[0pt]
&+\frac{u+cv}{u+cv+cy}\,\frac{\partial}{\partial
u}\,\Bigg(\int_{0}^{u+cv+cy}\P\big\{\Y{n+1}>z\big\}\,f_{Y}^{*n}(u+cv+cy-z)\,dz\,\Bigg),
\end{aligned}
\end{equation*}
where $\dfrac{\partial}{\partial
u}\Bigg(\dfrac{u+cv}{u+cv+cy}\Bigg)=\dfrac{cy}{(u+cv+cy)^2}$ and
\begin{equation*}
\begin{aligned}
\frac{\partial}{\partial
u}\,\Bigg(\int_{0}^{u+cv+cy}&\P\{\Y{n+1}>z\}\,f_{Y}^{*n}(u+cv+cy-z)\,dz\Bigg)
\\[0pt]
&=\int_{0}^{u+cv+cy}\P\big\{\Y{n+1}>z\big\}\,\Bigg(\frac{\partial}{\partial
u}\,f_{Y}^{*n}(u+cv+cy-z)\Bigg)\,dz
\\[0pt]
&+\P\big\{\Y{n+1}>u+cv+cy\big\}\,f_{Y}^{*n}(0),
\end{aligned}
\end{equation*}
For $n\geqslant 2$, differentiation of $n$-fold convolutions yields
\begin{equation*}
\begin{aligned}
\frac{\partial}{\partial u}f_{Y}^{*n}(u+cv+cy-z)&=\frac{\partial}{\partial
u}\int_{0}^{u+cv+cy-z}f_{Y}\big((u+cv+cy-z)-\zeta\big)f_{Y}^{*(n-1)}(\zeta)\,d\zeta
\\[0pt]
&=\underbrace{\int_{0}^{u+cv+cy-z}f_{Y}^{\,\prime}\left(\left(u+cv+cy-z\right)-\zeta\right)
f_{Y}^{*(n-1)}(\zeta)\,d\zeta}_{\int_{0}^{u+cv+cy-z}
f_{Y}^{\,\prime}(\xi)f_{Y}^{*(n-1)}(\left(u+cv+cy-z\right)-\xi)\,d\xi}
\\[0pt]
&+f_{Y}(0)f_{Y}^{*(n-1)}(u+cv+cy-z).
\end{aligned}
\end{equation*}
whence, by elementary calculations, the result.
\end{proof}

\section{Approximations in direct level crossing problem}\label{asrgrhger}

The probability density function (p.d.f.) and cumulative distribution function
(c.d.f.) of a Gaussian distribution with mean $a$ and variance $b^2$ are
denoted by $\UGauss{a}{b^2}(x)$ and $\Ugauss{a}{b^2}(x)$.

\subsection{Core integral expressions}\label{srdthrdjh}

For $t$, $u$, $c$, $M$, and $D^{2}$ fixed positive constants, the elementary
integral expressions are defined as
\begin{equation}\label{rterherher}
\Elem{u,c}{k}(t)\eqDef\int_{0}^{\frac{ct}{u}}\frac{1}{(x+1)^{k}}\,
\Ugauss{cM(x+1)}{\frac{c^{2}D^{2}}{u}(x+1)}(x)\,dx,\quad k=0,1,2,\dots.
\end{equation}
We write $\cS\eqDef M^{-1}$ and $\Elem{u,0}{k}(t)\eqDef\lim_{c\to
0}\Elem{u,c}{k}(t)$,
$\Elem{u,\infty}{k}(t)\eqDef\lim_{c\to\infty}\Elem{u,c}{k}(t)$, and the like.

These integral expressions can be expressed through c.d.f.
$F\big(x;\muIG,\lambdaIG,p\big)$ of a generalized inverse Gaussian
distribution\footnote{There are some differences in terminology. In
\cite{[Zigangirov=1962]}, this distribution is called Wald's distribution.
Several authors (see \cite{[Morlat=1956]}, \cite{[Seshadri=1997]},
\cite{[Perreault=1999]}) attribute the invention of generalized inverse
Gaussian distributions to E.\,Halphen and use the term ``Halphen Distribution
System'' or ``Halphen's laws''. The others, e.g., M.A.\,Chaudry and
S.M.\,Zubair \cite{[Chaudry=2002]}, refer to B.\,J{\o}rgensen
\cite{[Jorgensen=1982]} and attribute the invention of generalized inverse
Gaussian distribution to I.J.\,Good \cite{[Good=1953]}.}, which depends on
parameters $\muIG>0$, $\lambdaIG>0$, and $p\in\Rline$, and whose p.d.f.
is\footnote{Note that the choice $p=-{1}/{2}$ yields the ``ordinary'' inverse
Gaussian distribution.}
\begin{equation}\label{45t34y34}
\begin{aligned}
f\big(x;\muIG,\lambdaIG,p\big)
&\eqDef\frac{e^{-\frac{\lambdaIG}{\muIG}}}{2\muIG^p\BesselK{p}\Big(\frac{\lambdaIG}{\muIG}\Big)}\,x^{p-1}
\exp\Bigg\{-\frac{\lambdaIG(x-\muIG)^{2}}{2\muIG^{\,2} x}\Bigg\}
\\[0pt]
&=\frac{\sqrt{2\pi}\,e^{-\frac{\lambdaIG}{\muIG}}}{2\muIG^p\BesselK{p}\Big(\frac{\lambdaIG}{\muIG}\Big)}\,x^{p-1}
\,\Ugauss{0}{1}\left(\sqrt{\frac{\lambdaIG}{x}}\,\Bigg(\frac{x}{\muIG}-1\Bigg)\right),\quad
x>0,
\end{aligned}
\end{equation}
where $\BesselK{p}(z)$, $z>0$, with $p\in\Rline$, denotes the modified Bessel
function of the second kind. In particular, for $u>0$, $t>0$, we have
\begin{equation}\label{drgerherhe}
\Elem{u,c}{1}(t)=\begin{cases}
\Big(F\big(x+1;\muIG,\lambdaIG,-\tfrac{1}{2}\big)
\\[0pt]
\hskip 30pt -F\big(1;\muIG,\lambdaIG,-\tfrac{1}{2}\big)
\Big)\,\big|_{x=\frac{ct}{u},\muIG=\frac{1}{1-cM},\lambdaIG=\frac{u}{c^{2}D^{2}}},&
0<c<\cS,
\\[0pt]
\exp\bigg\{-\dfrac{2\lambdaIG}{\HmuIG}\bigg\}\,\Big(F\big(x+1;\HmuIG,\lambdaIG,-\tfrac{1}{2}\big)
\\[0pt]
\hskip 30pt
-F\big(1;\HmuIG,\lambdaIG,-\tfrac{1}{2}\big)\Big)\,\big|_{x=\frac{ct}{u},\HmuIG=\frac{1}{cM-1},
\lambdaIG=\frac{u}{c^{2}D^{2}}},&c>\cS,
\end{cases}
\end{equation}
and
\begin{equation}\label{asdfgreherXX}
\begin{aligned}
\Elem{u,0}{1}(t)&=\UGauss{0}{1}\left(\frac{M\sqrt{u}}{D}\right)-\UGauss{0}{1}\left(\frac{Mu-t}{D\sqrt{u}}\right),
\\[0pt]
\Elem{u,\cS}{1}(t)&=2\left(\UGauss{0}{1}\left(\dfrac{M\sqrt{u}}{D}\right)-\UGauss{0}{1}\left(\dfrac{Mu}{D\sqrt{u+\cS
t}}\right)\right),
\\[0pt]
\Elem{u,\infty}{1}(t)&=0.
\end{aligned}
\end{equation}

The c.d.f. of generalized inverse Gaussian distribution can be represented in
terms of c.d.f. and p.d.f. of a standard Gaussian distribution, e.g.,
\begin{equation}\label{xzfvbdbf}
F\big(x;\muIG,\lambdaIG,-\tfrac{1}{2}\big)=
\UGauss{0}{1}\left(\sqrt{\frac{\lambdaIG}{x}}\,\Bigg(\frac{x}{\muIG}-1\Bigg)\right)
+\exp\Bigg\{\frac{2\lambdaIG}{\muIG}\Bigg\}
\,\UGauss{0}{1}\left(-\sqrt{\frac{\lambdaIG}{x}}\,\Bigg(\frac{x}{\muIG}+1\Bigg)\right),\quad
x>0,
\end{equation}
and\footnote{We do not present here all the expressions for
$F\big(x;\muIG,\lambdaIG,\tfrac{1}{2}\big)$,
$F\big(x;\muIG,\lambdaIG,-\tfrac{3}{2}\big)$, for the derivatives like
$\frac{\partial}{\partial\lambdaIG}F\big(x;\muIG,\lambdaIG,-\tfrac{1}{2}\big)$
and
$\frac{\partial}{\partial\lambdaIG}F\big(x;\muIG,\lambdaIG,-\tfrac{1}{2}\big)$,
and for $\Elem{u,c}{0}(t)$, $\Elem{u,c}{2}(t)$, though they are available by
means of direct calculations and similar to these presented: this would require
dramatically more space. We leave this to the reader.}
\begin{equation*}
\begin{aligned}
F\big(x;\muIG,\lambdaIG,-\tfrac{5}{2}\big)
&=\UGauss{0}{1}\left(\sqrt{\frac{\lambdaIG}{x}}\,\bigg(\frac{x}{\muIG}-1\bigg)\right)
+\frac{\lambdaIG^{2}-3\lambdaIG\muIG+3\muIG^{\,2}}{\lambdaIG^{2}+3\lambdaIG\muIG+3\muIG^{\,2}}
\exp\Bigg\{\frac{2\lambdaIG}{\muIG}\Bigg\}
\\[0pt]
&\times\UGauss{0}{1}\left(-\sqrt{\frac{\lambdaIG}{x}}\,\bigg(\frac{x}{\muIG}+1\bigg)\right)
+\frac{2\sqrt{\lambdaIG}\,\muIG^{\,2}\,\big(\lambdaIG+3x\big)}{x^{3/2}
\,\big(\lambdaIG^{2}+3\lambdaIG\muIG+3\muIG^{\,2}\big)}\,
\Ugauss{0}{1}\left(\sqrt{\frac{\lambdaIG}{x}}\,\bigg(\frac{x}{\muIG}-1\bigg)\right),\quad
x>0.
\end{aligned}
\end{equation*}

\subsection{Inverse Gaussian approximation, as $u$ tends to infinity}\label{aesrfgeh}

The inverse Gaussian approximation for $\P\big\{\Tich{u,c}\leqslant t\big\}$ in
the direct level crossing problem was studied in
\cite{[Malinovskii=2017a]}--\cite{[Malinovskii=2018=dan=1]},
\cite{[Malinovskii=Malinovskii=2017]}. For $c\geqslant 0$, $u>0$, $0<v<t$,
$\cS\eqDef\E{Y}/\E{T}$, and for\footnote{For $T$ and $Y$ exponentially
distributed with parameters $\paramT>0$ and $\paramY>0$, we have
$\E{T}=\paramT^{-1}$, $\E{Y}=\paramY^{-1}$, $\D{T}=\paramT^{-2}$,
$\D{Y}=\paramY^{-2}$, and $M=\paramY/\paramT$,
$D^{\,2}=2\,\paramY/\paramT^{\,2}$, whence
$D/M^{\,3/2}=\sqrt{2\paramT}/\paramY$.}
\begin{equation}\label{adsfgherfjnfnm}
M\eqDef\E{T}/\E{Y},\quad
D^{\,2}\eqDef\big((\E{T})^{2}\D{Y}+(\E{Y})^{2}\D{T}\big)/(\E{Y})^{3}
\end{equation}
we write
\begin{equation*}
\AInt{M}{u,c}(t\mid v)\eqDef\int_{0}^{\frac{c(t-v)}{cv+u}}\frac{1}{x+1}
\,\Ugauss{cM(x+1)}{\frac{c^{2}D^{\,2}}{cv+u}(x+1)}(x)\,dx
\end{equation*}
and note that $\AInt{M}{u,c}(t)\eqDef\AInt{M}{u,c}(t\mid 0)$ equals
$\Elem{u,c}{1}(t)=\int_{0}^{\frac{ct}{u}}\frac{1}{x+1}\,
\Ugauss{cM(x+1)}{\frac{c^{2}D^{\,2}}{u}(x+1)}(x)\,dx$.

The following theorem for conditional distribution of $\Tich{u,c}$ (see the
left-hand side of \eqref{dfgbehredf}) is fundamental.

\begin{theorem}\label{srdthjrfX}
In the renewal model, let p.d.f. $f_{T}$ and $f_{Y}$ be bounded above by a
finite constant, $D^{\,2}>0$, $\E({T}^{3})<\infty$, $\E({Y}^{3})<\infty$. Then
for any fixed $c\geqslant 0$ and $0<v<t$ we have
\begin{equation*}
\sup_{t>v}\,\left|\,\P\big\{v<\Tich{u,c}\leqslant
t\mid\T{1}=v\big\}-\AInt{M}{u,c}(t\mid v)\,\right|
=\underline{O}\left(\frac{\ln\left(u+cv\right)}{u+cv}\right),
\end{equation*}
as\footnote{With $c$ and $v$ fixed, $u+cv\to\infty$ is trivially equivalent to
$u\to\infty$.} $u+cv\to\infty$.
\end{theorem}

The following results for non-conditional distribution of $\Tich{u,c}$ is an
easy corollary of Theorem~\ref{srdthjrfX} and equality \eqref{ewrktulertye}.

\begin{theorem}\label{er6u6rti76}
Suppose that conditions of Theorem~\ref{srdthjrfX} are satisfied. Then
\begin{equation*}
\sup_{t>0}\,\bigg|\,\P\big\{\Tich{u,c}\leqslant
t\big\}-\int_{0}^{t}\AInt{M}{u,c}(t\mid
v)\,f_{\T{1}}(v)\,dv\,\bigg|=\underline{O}\left(\frac{\ln u}{u}\right),\quad
u\to\infty.
\end{equation*}
\end{theorem}

We can replace the integral $\int_{0}^{t}\AInt{M}{u,c}(t\mid
v)\,f_{\T{1}}(v)\,dv$ by the integral
$\int_{0}^{t}\AInt{M}{u,c}(t-v)\,f_{\T{1}}(v)\,dv$, which is a convolution. It
agrees with the probabilistic intuition about the role which plays the first
time interval $\T{1}$ in the event of crossing a high level $u$ within finite
time $t$: given $\T{1}=v$, the whole time length becomes $t-v$, with no other
changes.

\begin{theorem}\label{rtutujrtirt6}
Suppose that conditions of Theorem~\ref{srdthjrfX} are satisfied, and that
$\E\T{1}<\infty$. Then
\begin{equation*}
\sup_{t>0}\,\bigg|\,\P\big\{\Tich{u,c}\leqslant t\big\}
-\int_{0}^{t}\AInt{M}{u,c}(t-v)\,f_{\T{1}}(v)\,dv\,\bigg|=\underline{O}\left(\frac{\ln
u}{u}\right),\quad u\to\infty.
\end{equation*}
\end{theorem}

\subsection{Approximation, as $u$ and $t$ tend to infinity}\label{wqreterjh}

Whereas the influence of $\T{1}$ in Theorem~\ref{rtutujrtirt6} can not be
eliminated for $t$ small and moderate, it becomes negligible for $t$ large.
Given that $t\to\infty$, the integral
$\int_{0}^{t}\AInt{M}{u,c}(t-v)\,f_{\T{1}}(v)\,dv$ in
Theorem~\ref{rtutujrtirt6} can be approximated by $\AInt{M}{u,c}(t)$, whence
the following result.

\begin{theorem}\label{esrytrf}
Suppose that conditions of Theorem~\ref{srdthjrfX} are satisfied, and that
$\E\T{1}<\infty$. Then
\begin{equation*}
\sup_{t>0}\,\Big|\,\P\big\{\Tich{u,c}\leqslant t\big\}
-\AInt{M}{u,c}(t)\,\Big|=\underline{O}\left(\frac{\ln
u}{u}\right)+\underline{O}\left(\frac{1}{t^{1/2}}\right),\quad t,u\to\infty.
\end{equation*}
\end{theorem}

\subsection{An outline of the proof}\label{srdtytujrt}

The proofs in \cite{[Malinovskii=2017a]}, \cite{[Malinovskii=2017b]} are
conducted in a uniform manner. In the proof of Theorem~\ref{srdthjrfX}, Step~0
is the identity \eqref{dfgbehredf}, or its copy \eqref{wertyrjhn}, for the
probability $\P\big\{v<\Tich{u,c}\leqslant t\mid\T{1}=v\big\}$. Step~1 is a
reduction of the range of integration in \eqref{dfgbehredf}. This cutting off
of unlikely events, such as when $\Y{n+1}$ is excessively large compared to the
whole sum $\sum_{i=1}^{n}\Y{i}$, complies with intuition. Step~2 is a reduction
of the range of summation in \eqref{dfgbehredf}. This applies Nagaev's
inequalities for sums used to reject the summands in the range $0<n<\EnOne$,
for which the probability of the event
$\{\homM{u+cv+cy}=n\}=\big\{\sum_{i=1}^{n}\Y{i}\leqslant
u+cv+cy<\sum_{i=1}^{n+1}\Y{i}\big\}$ is small, as $u+cv+cy$ is large. This
step is also intuitively clear. It relies on the fact that under mild technical
assumptions the occurrence of few renewals in a long time interval is an
unlikely event.

Step~3 consists in applying the Berry-Esseen bound with non-uniform remainder
term to $n$-fold convolutions $f_{Y}^{*n}$ and $f_{T}^{*n}$ in
\eqref{wertyrjhn}, where the ranges of integration and summation are reduced.
It is the nub of the proof, where the full force of the central limit theory is
applied. Step~4 is an estimation of the remaining term, and Step~5 is an
elaboration of the approximating term obtained in this way. In the course of
this study, certain identities (see, e.g., Section~4.3 in
\cite{[Malinovskii=2017a]}) are used, which allow us to represent sums in the
form of integral sums.

\section{Approximation for derivatives}\label{wqstgerhs}

The interest in derivatives $\frac{\partial}{\partial
c}\,\P\big\{\Tich{u,c}\leqslant t\big\}$ and $\frac{\partial}{\partial
u}\,\P\big\{\Tich{u,c}\leqslant t\big\}$ is related, e.g., with
Theorem~\ref{etrtjt}: equality \eqref{ewqrfghwre} is the basis of a standard
method for studying the monotony of an implicit function.

In the case we are considering, for each fixed $u$ the derivative
$\frac{\partial}{\partial c}\,\P\big\{\Tich{u,c}\leqslant t\big\}$, considered
as a function of $c\geqslant 0$, as well as for each fixed $c$ the derivative
$\frac{\partial}{\partial u}\,\P\big\{\Tich{u,c}\leqslant t\big\}$, considered
as a function of $u\geqslant 0$, are negative. This follows straightforwardly
from the definition of $\P\big\{\Tich{u,c}\leqslant t\big\}$: for each fixed
$u$, this function is monotone decreasing, as the variable $c\geqslant 0$
monotone increases; for each fixed $c$, this function is monotone decreasing,
as the variable $u\geqslant 0$ monotone increases.

We will study these derivatives much deeper. First, we examine their
approximations; this will allow us to draw conclusions about their magnitude.
Secondly, we develop a technique that allows us to study higher order
derivatives in the same way.

\subsection{Approximation for derivative $\frac{\partial}{\partial
c}\,\P\big\{\Tich{u,c}\leqslant t\big\}$}\label{adsfgvre}

\begin{theorem}\label{wertyujkl}
In the renewal model, let p.d.f. $f_{T}$ and differentiable $f_{Y}$ be bounded
above by a finite constant, $D^{\,2}>0$, $\E({T}^{3})<\infty$,
$\E({Y}^{3})<\infty$. Then for any fixed $c>0$ and $0<v<t$ we have
\begin{equation}\label{34t64y6e5yue}
\begin{aligned}
\frac{\partial}{\partial c}\,\P\big\{\Tich{u,c}\leqslant
t\big\}&=-\int_{0}^{t}f_{Y}(u+cv)\,v\,f_{\T{1}}(v)\,dv
\\[0pt]
&+\int_{0}^{t}\frac{\partial}{\partial c}\,\P\big\{v<\Tich{u,c}\leqslant
t\mid\T{1}=v\big\}\,f_{\T{1}}(v)\,dv,
\end{aligned}
\end{equation}
where
\begin{equation}\label{tyuturtfert}
\begin{aligned}
&\sup_{t>v}\;\Bigg|\;\frac{\partial}{\partial c}\,\P\big\{v<\Tich{u,c}\leqslant
t\mid\T{1}=v\big\}
\\[0pt]
&\hskip 60pt-\frac{M(u+cv)}{c^2D^2}\left((1-cM)\,\Elem{u,c}{0}(t\mid
v)-\Elem{u,c}{1}(t\mid v)\right)
\\[0pt]
&\hskip 60pt+\frac{Mu}{c^2D^2}\left((1-cM)\,\Elem{u,c}{1}(t\mid
v)-\Elem{u,c}{2}(t\mid v)\right)+\frac{1}{c}\,\Elem{u,c}{1}(t\mid v)
\\[0pt]
&\hskip 60pt-\frac{1}{c}\,\Elem{u,c}{2}(t\mid
v)\;\Bigg|=\underline{O}\left(\frac{\ln\left(u+cv\right)}{u+cv}\right),
\end{aligned}
\end{equation}
as $u+cv\to\infty$.
\end{theorem}

The starting point in the proof of Theorems~\ref{srdthjrfX} was Kendall's
identity, whereas the starting point in the proof of Theorem~\ref{wertyujkl} is
Lemma~\ref{yukiyulu}. The proof of Theorem~\ref{wertyujkl} follows the scheme
outlined in Section~\ref{srdtytujrt}. We leave the details to the reader.

In the same way as Theorem~\ref{esrytrf} follows from Theorem~\ref{srdthjrfX},
as $u,t\to\infty$, it follows from Theorem~\ref{wertyujkl} that the derivative
$\frac{\partial}{\partial c}\,\P\big\{\Tich{u,c}\leqslant t\big\}$ is
approximated, as $u,t\to\infty$, by the expression
\begin{equation}\label{dsfgrehgdfhd}
\begin{aligned}
F_{u,c}(t)&=\frac{Mu}{c^2D^2}\,\left((1-cM)\,\Elem{u,c}{0}(t)-\Elem{u,c}{1}(t)\right)
\\[0pt]
&-\frac{Mu}{c^2D^2}\left((1-cM)\,\Elem{u,c}{1}(t)-\Elem{u,c}{2}(t)\right)
-\frac{1}{c}\,\Elem{u,c}{1}(t)+\frac{1}{c}\,\Elem{u,c}{2}(t).
\end{aligned}
\end{equation}

Using the representation of $\Elem{u,c}{0}(t)$, $\Elem{u,c}{1}(t)$,
$\Elem{u,c}{2}(t)$, such as \eqref{drgerherhe}, we can verify that
\begin{equation}\label{wqe4ryttjfyjk}
(1-cM)\,\Elem{u,c}{0}(t)-\Elem{u,c}{1}(t)
=\underline{O}\left(u^{-1}\right),\quad u\to\infty,
\end{equation}
and that
\begin{equation}\label{esrtgeyheje}
(1-cM)\,\Elem{u,c}{1}(t)-\Elem{u,c}{2}(t)
=\underline{O}\left(u^{-1}\right),\quad u\to\infty.
\end{equation}

\begin{figure}[t]
\center{\includegraphics[scale=.8]{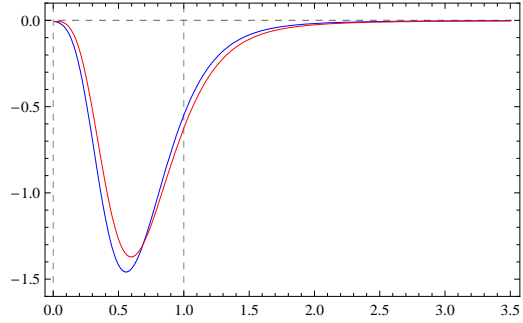}}
\caption{\small Graphs ($X$-axis is $c$) of the functions $F_{u,c}(t)$ (red),
$\AInt{M}{u,c}^{\,(0,1)}(t)$ (blue). Here $t=100$, $u=40$, $M=1$,
$D^2=6$.}\label{asrtdyujkl}
\end{figure}

\begin{remark}\label{wqe45r6u7rui}
Let us compare \eqref{dsfgrehgdfhd} with
$\AInt{M}{u,c}^{\,(0,1)}(t)\eqDef\frac{\partial}{\partial c}\AInt{M}{u,c}(t)$.
In other words, let us compare the ``approximation of derivative'' with the
``derivative of approximation''.

For $u>0$, $t>0$, $c>0$, we have by straightforward differentiation
\begin{equation*}
\begin{aligned}
\AInt{M}{u,c}^{\,(0,1)}(t)&=\frac{u\,(1-cM)}{c^{3}D^{2}}\,\Elem{u,c}{\,0}(t)
-\frac{1}{c}\,\Bigg(\frac{u\,(2-cM)}{c^{2}D^{2}}+1\Bigg)\,\Elem{u,c}{1}(t)
\\[-2pt]
&+\frac{u}{c^{3}D^{2}}\,\Elem{u,c}{\,2}(t)+\frac{t}{u+ct}\,
\Ugauss{cM(1+\frac{ct}{u})}{\frac{c^{2}D^{2}}{u}(1+\frac{ct}{u})}\bigg(\frac{ct}{u}\bigg).
\end{aligned}
\end{equation*}
Alternatively, this equality is written as
\begin{equation}\label{dfyjhkghkgh}
\begin{aligned}
\AInt{M}{u,c}^{\,(0,1)}(t)&=\frac{u}{c^{3}D^{2}}\left((1-cM)\,\Elem{u,c}{0}(t)-\Elem{u,c}{1}(t)\right)
\\[2pt]
&-\frac{u}{c^{3}D^{2}}\left((1-cM)\,\Elem{u,c}{1}(t)-\Elem{u,c}{2}(t)\right)
\\[0pt]
&-\frac{1}{c}\,\Elem{u,c}{1}(t)+\frac{t}{u+ct}\,
\Ugauss{cM(1+\frac{ct}{u})}{\frac{c^{2}D^{2}}{u}(1+\frac{ct}{u})}\bigg(\frac{ct}{u}\bigg).
\end{aligned}
\end{equation}
This equality\footnote{Bear in mind the asymptotic relations
\eqref{wqe4ryttjfyjk} and \eqref{esrtgeyheje}.} is suitable for comparison with
equality \eqref{dsfgrehgdfhd}. Both are illustrated in Fig.~\ref{asrtdyujkl}.
\end{remark}

We conclude this analysis with the following summary. The proximity between
$F_{u,c}(t)$, i.e., ``approximation of derivative'', and
$\AInt{M}{u,c}^{\,(0,1)}(t)$, i.e., ``derivative of approximation'',
illustrated numerically in Fig.~\ref{asrtdyujkl}, can be proved rigorously
using equalities \eqref{dsfgrehgdfhd} and \eqref{dfyjhkghkgh}, evaluated
analytically. However, the ``approximation of derivative'' is one thing and the
``derivative of approximation'' is another. Their study requires a separate
analysis; the naive belief that one can replace the other is largely
groundless.

\subsection{Approximation for derivative $\frac{\partial}{\partial
u}\,\P\big\{\Tich{u,c}\leqslant t\big\}$}\label{tyjtrujrtjurtj}

\begin{theorem}\label{qewrweytery}
In the renewal model, let p.d.f. $f_{T}$ and differentiable $f_{Y}$ be bounded
above by a finite constant, $D^{\,2}>0$, $\E({T}^{3})<\infty$,
$\E({Y}^{3})<\infty$. Then for any fixed $c>0$ and $0<v<t$ we have
\begin{equation}\label{retrehuerhr}
\begin{aligned}
\frac{\partial}{\partial u}\,\P\big\{\Tich{u,c}\leqslant
t\big\}&=-\int_{0}^{t}f_{Y}(u+cv)\,v\,f_{\T{1}}(v)\,dv
\\[0pt]
&+\int_{0}^{t}\frac{\partial}{\partial u}\,\P\big\{v<\Tich{u,c}\leqslant
t\mid\T{1}=v\big\}\,f_{\T{1}}(v)\,dv,
\end{aligned}
\end{equation}
where
\begin{equation}\label{dthrtjhurt}
\begin{aligned}
&\sup_{t>v}\;\Bigg|\;\frac{\partial}{\partial u}\,\P\big\{v<\Tich{u,c}\leqslant
t\mid\T{1}=v\big\}
\\[0pt]
&\hskip 60pt-\frac{M}{c\,D^2}\left((1-cM)\,\Elem{u,c}{1}(t\mid
v)-\Elem{u,c}{2}(t\mid v)\right)
\\[0pt]
&\hskip 60pt-\frac{1}{u}\left(\Elem{u,c}{1}(t\mid v)-\Elem{u,c}{2}(t\mid
v)\right)
\,\Bigg|=\underline{O}\left(\frac{\ln\left(u+cv\right)}{(u+cv)^2}\right),
\end{aligned}
\end{equation}
as $u+cv\to\infty$.
\end{theorem}

In the same way as in Section~\ref{adsfgvre}, we can deduce from
Theorem~\ref{qewrweytery} that the derivative $\frac{\partial}{\partial
u}\,\P\big\{\Tich{u,c}\leqslant t\big\}$ is approximated, as $u,t\to\infty$, by
the expression\footnote{Bear in mind the asymptotic relation
\eqref{esrtgeyheje}.}
\begin{equation}\label{fgyjfgkmhh}
G_{u,c}(t)=\frac{M}{c\,D^2}\left((1-cM)\,\Elem{u,c}{1}(t)-\Elem{u,c}{2}(t)\right)
+\frac{1}{u}\left(\Elem{u,c}{1}(t)-\Elem{u,c}{2}(t)\right).
\end{equation}

\begin{figure}[t]
\center{\includegraphics[scale=.8]{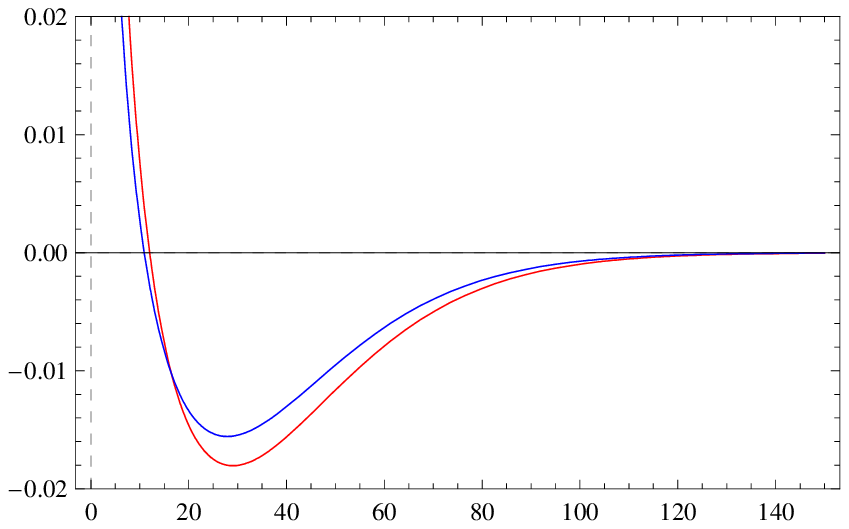}\\
\includegraphics[scale=.8]{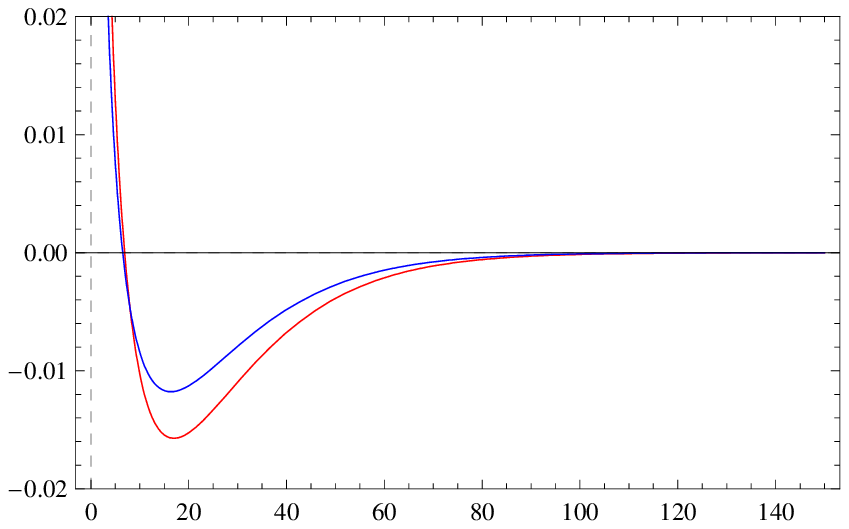}}
\caption{\small Graphs ($X$-axis is $u$) of the functions $G_{u,c}(t)$ (red),
$\AInt{M}{u,c}^{\,(1,0)}(t)$ (blue). Here $t=100$, $M=1$, $D^2=6$,
$c=0{.}8<\cS$ (above), $c=1{.}2>\cS$ (below), where $\cS=1$.}\label{yt7yiktyyu}
\end{figure}

\begin{remark}\label{aesrgegher}
Let us compare \eqref{fgyjfgkmhh} with
$\AInt{M}{u,c}^{\,(1,0)}(t)\eqDef\frac{\partial}{\partial u}\AInt{M}{u,c}(t)$.
In other words, let us compare the ``approximation of derivative'' with the
``derivative of approximation''.

For $u>0$, $t>0$, $c>0$, we have by straightforward differentiation
\begin{equation*}
\begin{aligned}
\AInt{M}{u,c}^{\,(1,0)}(t)&=-\frac{(1-cM)^{2}}{2c^{2}D^{2}}\,\Elem{u,c}{\,0}(t)+\frac{1}{2u}\,\Elem{u,c}{1}(t)
+\frac{(1-cM)}{c^2D^{2}}\,\Elem{u,c}{1}(t)
\\[0pt]
&-\frac{1}{2c^{2}D^{2}}\,\Elem{u,c}{2}(t)-\frac{ct}{u\,(u+ct)}\,\Ugauss{cM(1+\frac{ct}{u})}{\frac{c^{2}D^{2}}{u}
(1+\frac{ct}{u})}\bigg(\frac{ct}{u}\bigg).
\end{aligned}
\end{equation*}
Alternatively, this equality is written as
\begin{equation}\label{qwertyerjk}
\begin{aligned}
\AInt{M}{u,c}^{\,(1,0)}(t)&=\frac{1}{c^2D^{2}}\left((1-cM)\,\Elem{u,c}{1}(t)-\Elem{u,c}{2}(t\mid
v)\right)
\\[2pt]
&-\frac{1}{2c^2D^{2}}\left((1-cM)^{2}\,\Elem{u,c}{\,0}(t)-\Elem{u,c}{\,2}(t)\right)
\\[-2pt]
&+\frac{1}{2u}\,\Elem{u,c}{1}(t)-\frac{ct}{u\,(u+ct)}\,\Ugauss{cM(1+\frac{ct}{u})}{\frac{c^{2}D^{2}}{u}
(1+\frac{ct}{u})}\bigg(\frac{ct}{u}\bigg).
\end{aligned}
\end{equation}
This equality is suitable for comparison with equality \eqref{fgyjfgkmhh}. Both
are illustrated in Fig.~\ref{yt7yiktyyu}.
\end{remark}

\begin{remark}\label{w45ey6u5}
At the beginning of Section~\ref{wqstgerhs}, we noted that the derivative
$\frac{\partial}{\partial u}\,\P\big\{\Tich{u,c}\leqslant t\big\}$ is negative
\emph{for all} $u>0$. But neither the approximation $G_{u,c}(t)$ nor the
expression $\AInt{M}{u,c}^{\,(1,0)}(t)$ are negative for all $u>0$; for $u$
small and moderate (see Fig.~\ref{yt7yiktyyu}), both these expressions take
positive values. This is not a flaw, because the approximation of
Theorem~\ref{qewrweytery} only works for $u$ large.
\end{remark}

\subsection{Approximation for higher-order derivatives}

The approximations for higher-order derivatives, such as
$\frac{\partial^2}{\partial c^2}\,\P\big\{\Tich{u,c}\leqslant t\big\}$,
$\frac{\partial^2}{\partial u^2}\,\P\big\{\Tich{u,c}\leqslant t\big\}$, and
$\frac{\partial^2}{\partial c\,\partial u}\,\P\big\{\Tich{u,c}\leqslant
t\big\}$ used in \eqref{ewryjhnmfm} are carried out as described above, and
left to the reader.

\section{Approximations in inverse level crossing problem}\label{srgterjer}

\subsection{Structural results for fixed-probability level}\label{werdtyjmt}

The inverse level crossing problem when $T$ and $Y$ are exponentially
distributed with parameters $\paramT>0$ and $\paramY>0$ was studied in
\cite{[Malinovskii=2012]} and \cite{[Malinovskii=2014b]}. In a diffusion model,
the similar analysis was done in \cite{[Malinovskii=2014a]}; see also
\cite{[Malinovskii=2009]}. When $T$ and $Y$ are non-exponentially distributed,
simulation analysis in the inverse level crossing problem was done in
\cite{[Malinovskii=Kosova=2014]}.

Going along the road map set in \cite{[Malinovskii=2009]}--
\cite{[Malinovskii=2014b]}, we focus on the asymptotic structure of the
fixed-probability level. We start with the following simple result.

\begin{theorem}\label{gyuiytik}
Assume that $f_{T}(x)$ and $f_{Y}(x)$ are bounded above by a finite constant,
$D^{2}>0$, $\E({T}^{3})<\infty$, and $\E({Y}^{3})<\infty$. Then for
$t\to\infty$, we have
\begin{equation}\label{3etu6juh}
\Ruat(0)=\frac{t}{M}+\frac{D}{M^{\,3/2}}\,\NQuant{\alpha}\sqrt{t}\,\big(1+\overline{o}(1)\big).
\end{equation}
\end{theorem}

Since $\Ruat(0)$ is a solution to the equation $\P\big\{\Tich{u,0}\leqslant
t\big\}=\alpha$, where $\Tich{u,0}\eqDef\inf\left\{s>0:\homV{s}>u\right\}$, and
since the trajectories of the compound renewal process $\homV{s}$, $s\geqslant
0$, are (a.s.) step functions with only jumps up, this equation rewrites as
$\P\left\{\homV{t}>u\right\}=\alpha$. Therefore, Theorem~\ref{gyuiytik} is a
direct corollary of the normal approximation for the distribution of compound
renewal process $\homV{t}$, which is well known: as $t\to\infty$, the
probability $\P\left\{\homV{t}>u\right\}$ is approximated by
\begin{equation*}
1-\UGauss{0}{1}\left(\frac{u-\E\homV{t}}{\sqrt{\,\D\homV{t}}}\right),
\end{equation*}
where
\begin{equation}\label{sdrthrjh}
\begin{aligned}
\E{\homV{t}}&=(\E{Y}/\E{T})\,t+\E{Y}(\D{T}-(\E{T})^{\,2})/(2(\E{T})^{\,2})+\overline{o}(1),
\\[0pt]
\D{\homV{t}}&=(((\E{Y})^{\,2}\D{T}+(\E{T})^{\,2}\D{Y})/(\E{Y})^3)\,t+\overline{o}(t),
\end{aligned}
\end{equation}
whence \eqref{3etu6juh}.

The following theorem is a generalization of Theorem~2.2 in
\cite{[Malinovskii=2014b]}; it is worthwhile to compare it with Theorem~1 in
\cite{[Malinovskii=2014a]}.

\begin{theorem}\label{dtfyjfgjk}
Assume that $f_{T}(x)$ and $f_{Y}(x)$ are bounded above by a finite constant,
$D^{2}>0$, $\E({T}^{3})<\infty$, and $\E({Y}^{3})<\infty$. Then for
$t\to\infty$, we have
\begin{equation}\label{eryhjmgh}
\Ruat(\cS)=\dfrac{D}{M^{\,3/2}}\,\NQuant{\alpha/2}\sqrt{t}\,\big(1+\overline{o}(1)\big).
\end{equation}
\end{theorem}

\begin{proof}[Proof of Theorem~\ref{dtfyjfgjk}]\label{yukryikt}
The left-hand side of equation \eqref{rtyujtkjtyk} is (see
\eqref{ewrktulertye})
\begin{equation*}
\P\left\{\Tich{u,c}\leqslant
t\right\}=\int_{0}^{t}\P\left\{u+cv-Y<0\right\}\,f_{\T{1}}(v)\,dv
+\int_{0}^{t}\P\left\{v<\Tich{u,c}\leqslant
t\mid\T{1}=v\right\}\,f_{\T{1}}(v)\,dv,
\end{equation*}
whence for $c=\cS$ \eqref{rtyujtkjtyk} rewrites as
\begin{equation}\label{ewtyjhjmfg}
\int_{0}^{t}\P\left\{u+\cS v<\Y{1}\right\}f_{\T{1}}(v)\,dv
+\int_{0}^{t}\P\left\{v<\Tich{u,\cS}\leqslant
t\mid\T{1}=v\right\}f_{\T{1}}(v)\,dy=\alpha.
\end{equation}
Bearing in mind the second equality in \eqref{asdfgreherXX}, the probability
$\P\left\{v<\Tich{u,\cS}\leqslant t\mid\T{1}=v\right\}$ is approximated by
\begin{equation*}
2\left(\UGauss{0}{1}\left(\sqrt{\frac{\E{T}(v\,\E{Y}+u\,\E{T})}{(\E{Y})^{2}D^{2}}}\,\right)
-\UGauss{0}{1}\left(\sqrt{\frac{\E{T}}{(\E{Y})^{2}D^{2}}}
\frac{u\,\E{T}+v\,\E{Y}}{\sqrt{u\,\E{T}+t\,\E{Y}}}\right)\right).
\end{equation*}

Let us show that $\Ruat(\cS)$ in \eqref{eryhjmgh} is an asymptotic solution to
equation \eqref{ewtyjhjmfg}. First, bearing in mind that $\E{T}^{3}<\infty$,
$\E{Y}^{3}<\infty$, it is easily seen that
\begin{equation*}
\int_{0}^{t}\P\{u+\cS v-\Y{1}<0\}f_{\T{1}}(v)\,dv\to 0,\quad t\to\infty,\
u\to\infty.
\end{equation*}
Secondly, it is easy to see that
\begin{equation*}
\UGauss{0}{1}\left(\sqrt{\frac{\E{T}(v\,\E{Y}+u\,\E{T})}{(\E{Y})^{2}D^{2}}}\,\right)\to
1,\quad u\to\infty.
\end{equation*}
Selecting in \eqref{eryhjmgh} $u$ as $\underline{O}\left(t^{1/2}\right)$, we
have
\begin{equation*}
\frac{u\,\E{T}+v\,\E{Y}}{\sqrt{u\,\E{T}+t\,\E{Y}}}
=\frac{u\,\E{T}}{\sqrt{t\,\E{Y}}}\,\big(1+\overline{o}(1)\big),\quad
t\to\infty.
\end{equation*}
Therefore, equation \eqref{ewtyjhjmfg} reduces to
\begin{equation*}
2\left(1-\UGauss{0}{1}\left(\sqrt{\frac{\E{T}}{(\E{Y})^{2}D^{2}}}
\frac{u\,\E{T}}{\sqrt{t\,\E{Y}}}\right)\right)=\alpha,
\end{equation*}
whence the result.
\end{proof}

The following theorem is a generalization of Theorem~2.2 in
\cite{[Malinovskii=2014b]}. It is useful to compare it with Theorem~2 in
\cite{[Malinovskii=2014a]}, or Theorem~4.4 in \cite{[Malinovskii=2009]}.

\begin{theorem}\label{ergtrewghwerg}
Assume that differentiable  $f_{T}(x)$ and $f_{Y}(x)$ are bounded above by a
finite constant, and $D^{2}>0$, $\E({T}^{3})<\infty$, $\E({Y}^{3})<\infty$.
Then for $\cS=M^{-1}$ we have
\begin{equation}\label{asdfgrfjhf}
\Ruat(c)=\begin{cases} \big(\cS-c\big)\,t+\dfrac{D}{M^{\,3/2}}\,\funU{\alpha,t}
\left(\dfrac{M^{\,3/2}(\cS-c)}{D}\sqrt{t}\,\right)\sqrt{t}, &0\leqslant
c\leqslant\cS,
\\[10pt]
\dfrac{D}{M^{\,3/2}}\,\funU{\alpha,t}
\left(\dfrac{M^{\,3/2}(\cS-c)}{D}\sqrt{t}\,\right)\sqrt{t},&c>\cS,
\end{cases}
\end{equation}
where for $t$ sufficiently large the function $\funU{\alpha,t}(y)$,
$y\in\Rline$, is continuous and monotone increasing, as $y$ increases from
$-\infty$ to $0$, and monotone decreasing, as $y$ increases from $0$ to
$\infty$, and such that\footnote{We recall that
$0<\NQuant{\alpha}<\NQuant{\alpha/2}$ for $0<\alpha<\frac12$.}
\begin{equation*}
\lim_{y\to-\infty}\funU{\alpha,t}(y)=0,\
\lim_{y\to\infty}\funU{\alpha,t}(y)=\NQuant{\alpha}
\end{equation*}
and $\funU{\alpha,t}(0)=\NQuant{\alpha/2}\,\big(1+\overline{o}(1)\big)$, as
$t\to\infty$.
\end{theorem}

\begin{proof}[Proof of Theorem~\ref{ergtrewghwerg}]\label{fgjhrtjkrkj}
This proof is carried out in two stages. In each stage we make a suitable
change of variables. Its aim is to focus on the function $\funU{\alpha,t}(y)$,
$y\in\Rline$, and to check its monotony using the standard criterion based on
the sign of its derivative; this is calculated by means of (see
Theorem~\ref{etrtjt}) the implicit function derivative theorem.

\par\textbf{\emph{Step~1.}}
Let us consider the case $0<c<\cS=M^{-1}$. Regarding
equation~\eqref{rtyujtkjtyk}, we switch from the variables $u$ and $c$\; in its
left-hand side to the variables
\begin{equation}\label{serdthyntm}
z=\frac{u}{DM^{-3/2}\sqrt{t}}-\frac{(M^{-1}-c)\sqrt{t}}{DM^{-3/2}}\in\Rline,\quad
y=\frac{(M^{-1}-c)\sqrt{t}}{DM^{-3/2}}>0.
\end{equation}

The original equation \eqref{rtyujtkjtyk} rewrites as
\begin{equation*}
\P\big\{\Tich{u,c}\leqslant t\big\}\,\big|_{u=\frac{D\sqrt{t}}{M^{\,3/2}}(z+y),
c=\frac{1}{M}-\frac{D}{M^{\,3/2}\sqrt{t}}y} =\alpha.
\end{equation*}

To prove that $\funU{\alpha,t}(y)$, $y>0$, is monotone decreasing, we have to
prove that $\frac{d}{dy}\,\funU{\alpha,t}(y)<0$, $y>0$. Referring to the
implicit function derivative theorem (see Theorem~\ref{etrtjt}), we have
\begin{equation}\label{retyjrkrt}
\left.\frac{d}{dy}\,\funU{\alpha,t}(y)=-\left(\frac{\frac{\partial}{\partial
y}\Bigg(\P\big\{\Tich{u,c}\leqslant
t\big\}\,\Big|_{u=\frac{D\sqrt{t}}{M^{\,3/2}}(z+y),
c=\frac{1}{M}-\frac{D}{M^{\,3/2}\sqrt{t}}y} \Bigg)}{\frac{\partial}{\partial
z}\Bigg(\P\big\{\Tich{u,c}\leqslant
t\big\}\,\Big|_{u=\frac{D\sqrt{t}}{M^{\,3/2}}(z+y),
c=\frac{1}{M}-\frac{D}{M^{\,3/2}\sqrt{t}}y}\Bigg)}\right)
\;\right|_{z=\funU{\alpha,t}(y)}.
\end{equation}
The numerator is
\begin{equation}\label{yujmyhmk}
\begin{aligned}
&\frac{\partial}{\partial y}\left(\P\big\{\Tich{u,c}\leqslant
t\big\}\,\big|_{u=\frac{D\sqrt{t}}{M^{\,3/2}}\,(z+y),
c=\frac{1}{M}-\frac{D}{M^{\,3/2}\sqrt{t}}y}\right)
\\
&=\Bigg(\frac{\partial}{\partial u}\P\big\{\Tich{u,c}\leqslant
t\big\}\Bigg)\,\Bigg|_{u=\frac{D\sqrt{t}}{M^{\,3/2}}\,(z+y),
c=\frac{1}{M}-\frac{D}{M^{\,3/2}\sqrt{t}}y}\,\underbrace{\frac{\partial}{\partial
y}\,\Bigg(\frac{D\sqrt{t}}{M^{\,3/2}}\,(z+y)\Bigg)}_{\frac{D\sqrt{t}}{M^{\,3/2}}}
\\
&+\Bigg(\frac{\partial}{\partial c}\P\big\{\Tich{u,c}\leqslant
t\big\}\Bigg)\,\bigg|_{u=\frac{D\sqrt{t}}{M^{\,3/2}}\,(z+y),
c=\frac{1}{M}-\frac{D}{M^{\,3/2}\sqrt{t}}y}\underbrace{\frac{\partial}{\partial
y}\,\Bigg(\frac{1}{M}-\frac{D}{M^{\,3/2}\sqrt{t}}\,y\,\Bigg)}_{-\frac{D}{M^{\,3/2}\sqrt{t}}},
\end{aligned}
\end{equation}
whose approximation, as $t\to\infty$, follows from \eqref{fgyjfgkmhh},
\eqref{dsfgrehgdfhd}, \eqref{esrtgeyheje}, \eqref{wqe4ryttjfyjk}; for large $t$
this is obviously negative. The denominator is
\begin{equation}\label{ergtewgwergy}
\begin{aligned}
&\frac{\partial}{\partial z}\left(\P\big\{\Tich{u,c}\leqslant
t\big\}\,\big|_{u=\frac{D\sqrt{t}}{M^{\,3/2}}\,(z+y),
c=\frac{1}{M}-\frac{D}{M^{\,3/2}\sqrt{t}}y}\right)
\\
&=\Bigg(\frac{\partial}{\partial u}\P\big\{\Tich{u,c}\leqslant
t\big\}\Bigg)\,\Bigg|_{u=\frac{D\sqrt{t}}{M^{\,3/2}}\,(z+y),
c=\frac{1}{M}-\frac{D}{M^{\,3/2}\sqrt{t}}y}\,\underbrace{\frac{\partial}{\partial
z}\,\Bigg(\frac{D\sqrt{t}}{M^{\,3/2}}\,(z+y)\Bigg)}_{\frac{D\sqrt{t}}{M^{\,3/2}}},
\end{aligned}
\end{equation}
whose approximation, as $t\to\infty$, follows from \eqref{fgyjfgkmhh},
\eqref{esrtgeyheje}, \eqref{wqe4ryttjfyjk}; for large $t$ this is obviously
negative, whence the result.

\par\textbf{\emph{Step~2.}}
We continue the proof with investigating the case $c>\cS=M^{-1}$. We switch
from the variables $u$ and $c$ to the variables
\begin{equation}\label{drtgrfjhnrtjhntr}
z=\dfrac{u}{DM^{-3/2}\sqrt{t}}>0,\quad
y=\frac{(M^{-1}-c)\sqrt{t}}{DM^{-3/2}}<0,
\end{equation}

Let us rewrite the original equation \eqref{rtyujtkjtyk} as
\begin{equation*}
\P\big\{\Tich{u,c}\leqslant t\big\}\,\Big|_{u=\frac{D\sqrt{t}}{M^{\,3/2}}z,
c=\frac{1}{M}-\frac{D}{M^{\,3/2}\sqrt{t}}y}=\alpha.
\end{equation*}

To prove that $\funU{\alpha,t}(y)$, $y<0$, is monotone increasing, we have to
prove that $\frac{d}{dy}\,\funU{\alpha,t}(y)>0$, $y<0$. Referring to the
implicit function derivative theorem (see Theorem~\ref{etrtjt}), we have
\begin{equation*}
\left.\frac{d}{dy}\,\funU{\alpha,t}(y)=-\frac{\frac{\partial}{\partial
y}\P\big\{\Tich{u,c}\leqslant
t\big\}\,\Big|_{u=\frac{D\sqrt{t}}{M^{\,3/2}}\,z,\,
c=\frac{1}{M}-\frac{D}{M^{\,3/2}\sqrt{t}}y}}{\frac{\partial}{\partial
z}\P\big\{\Tich{u,c}\leqslant
t\big\}\,\Big|_{u=\frac{D\sqrt{t}}{M^{\,3/2}}\,z,\,
c=\frac{1}{M}-\frac{D}{M^{\,3/2}\sqrt{t}}y}}\;\right|_{z=\funU{\alpha,t}(y)}.
\end{equation*}
The numerator is
\begin{equation}\label{tryjutikty}
\begin{aligned}
&\frac{\partial}{\partial y}\P\big\{\Tich{u,c}\leqslant
t\big\}\,\bigg|_{u=\frac{D\sqrt{t}}{M^{\,3/2}}\,z,
c=\frac{1}{M}-\frac{D}{M^{\,3/2}\sqrt{t}}y}
\\
&=\Bigg(\frac{\partial}{\partial c}\P\big\{\Tich{u,c}\leqslant
t\big\}\Bigg)\,\bigg|_{u=\frac{D\sqrt{t}}{M^{\,3/2}}\,z,\,
c=\frac{1}{M}-\frac{D}{M^{\,3/2}\sqrt{t}}y}\frac{\partial}{\partial
y}\,\Bigg(\frac{1}{M}-\frac{D}{M^{\,3/2}\sqrt{t}}\,y\,\Bigg),
\end{aligned}
\end{equation}
whose approximation, as $t\to\infty$, follows from \eqref{dsfgrehgdfhd},
\eqref{esrtgeyheje}, \eqref{wqe4ryttjfyjk}; for large $t$ this is obviously
negative. The denominator is
\begin{equation}\label{yuktykt}
\begin{aligned}
&\frac{\partial}{\partial z}\P\big\{\Tich{u,c}\leqslant
t\big\}\,\bigg|_{u=\frac{D\sqrt{t}}{M^{\,3/2}}\,z,\,
c=\frac{1}{M}-\frac{D}{M^{\,3/2}\sqrt{t}}y}
\\
&=\Bigg(\frac{\partial}{\partial u}\P\big\{\Tich{u,c}\leqslant
t\big\}\Bigg)\,\Bigg|_{u=\frac{D\sqrt{t}}{M^{\,3/2}}\,z,\,
c=\frac{1}{M}-\frac{D}{M^{\,3/2}\sqrt{t}}y}\,\frac{\partial}{\partial
z}\,\Bigg(\frac{D\sqrt{t}}{M^{\,3/2}}\,z\,\Bigg).
\end{aligned}
\end{equation}
whose approximation, as $t\to\infty$, follows from \eqref{fgyjfgkmhh},
\eqref{esrtgeyheje}, \eqref{wqe4ryttjfyjk}; for large $t$ this is obviously
negative, whence the result.
\end{proof}

\subsection{Monotony and convexity of fixed-probability level}\label{defyguyl}

The fixed-probability level $\Ruat(c)$, $c\geqslant 0$, \emph{for all} $t$
monotone decreases, as $c$ increases. The following result, called weak-form
convexity, differs in that it is established by our means only for $t$ large.

\begin{theorem}[Weak-form convexity]\label{ew5tuy65}
Suppose that conditions of Theorem \ref{ergtrewghwerg} are satisfied. Then for
$t>0$ sufficiently large, the function $\Ruat(c)$, $c\geqslant\cS$, is convex.
\end{theorem}

The proof of Theorem~\ref{ew5tuy65} requires dramatically large space and is
left to the reader. Nevertheless, it is quite clear\footnote{In the diffusion
model, the proof of convexity was carried out with complete details in
\cite{[Malinovskii=2014a]}.}: one should check that for $c\geqslant\cS$ the
inequality $\frac{d^2}{dc^2}\Ruat(c)>0$ holds for $t$ sufficiently large. This
starts with equality \eqref{ewryjhnmfm}, proceeds with, first, calculation of
the second-order derivatives as it is done in Section~\ref{rgerhyryttryjh},
and, second, approximating them as it is done in Section~\ref{wqstgerhs}. The
proof of positivity of the approximation for $\frac{d^2}{dc^2}\Ruat(c)>0$, when
$t$ is large, brings the proof to a close.

\subsection{Heuristic fixed-probability level}\label{therhrher}

Let us introduce $\Huat(c)$, $c\geqslant 0$, which is a positive solution to
the equation\footnote{Recall that $\AInt{M}{u,c}(t)$ is an alternative notation
for $\Elem{u,c}{1}(t)=\int_{0}^{\frac{ct}{u}}\frac{1}{x+1}\,
\Ugauss{cM(x+1)}{\frac{c^{2}D^{\,2}}{u}(x+1)}(x)\,dx$.}
\begin{equation}\label{edthr6j}
\AInt{M}{u,c}(t)=\alpha,
\end{equation}
whose right-hand side is expressed\footnote{See \eqref{rterherher} with
$p=-{1}/{2}$, \eqref{drgerherhe}, and \eqref{xzfvbdbf}.} in a closed form.
Plainly, to get \eqref{edthr6j}, we replaced the left-hand side of the original
equation \eqref{rtyujtkjtyk} by an approximation found in
Theorem~\ref{esrytrf}.

\begin{theorem}\label{drtyryhrXX}
For\, $0\leqslant c<K\cS$, $0<K<1$, we have\footnote{For $c=0$, asymptotic
equality \eqref{adsfegh} coincides with \eqref{3etu6juh}.}
\begin{equation}\label{adsfegh}
\Huat(c)=\left(\cS-c\right)t+\frac{D}{M^{\,3/2}}\,\NQuant{\alpha}\sqrt{t}\big(1+\overline{o}(1)\big),\quad
t\to\infty.
\end{equation}
\end{theorem}

\begin{proof}[Proof of Theorem~\ref{drtyryhrXX}]
For\, $0\leqslant c<K\cS$, $0<K<1$, and
$u_{t}(c)=\left(\cS-c\right)t+\frac{D}{M^{\,3/2}}\sqrt{t}\,z_{t}(c)$, where
$z_{t}(c)=\underline{O}(1)$, as $t\to\infty$, we have
\begin{equation*}
\AInt{M}{u,c}(t)\,\big|_{\,u=u_t(c)}\sim
1-\UGauss{0}{1}\left(z_{t}(c)\right),\quad t\to\infty.
\end{equation*}
Therefore, the equation $\AInt{M}{u,c}(t)\,|_{\,u=u_{t}(c)}=\alpha$ rewrites as
$1-\UGauss{0}{1}\left(z_{t}(c)\right)=\alpha\big(1+\overline{o}(1)\big)$,
$t\to\infty$, whose solution is
$z_{t}(c)=\NQuant{\alpha}\big(1+\overline{o}(1)\big)$, $t\to\infty$.
\end{proof}

\begin{theorem}\label{tyukytkXX}
For $c_{t,\delta}=\cS-\frac{D}{M^{\,3/2}}\,\delta\,t^{-1/2}$,
$0\leqslant\delta<K$, we have
\begin{equation}\label{erty5yeh}
\Huat(c_{t,\delta})=\frac{D}{M^{\,3/2}}\,x_{\alpha}(\delta)\sqrt{t}\big(1+\overline{o}(1)\big),\quad
t\to\infty,
\end{equation}
where $x_{\alpha}(\delta)$ is a solution to the equation
\begin{equation}\label{srdty54urt}
1-\UGauss{0}{1}\left(-\delta+x\,\right)+\UGauss{0}{1}\left(-\delta-x\,\right)
\exp\left\{\,2\,\delta x\,\right\}=\alpha.
\end{equation}
\end{theorem}

\begin{proof}[Proof of Theorem~\ref{tyukytkXX}]
For $c_{t,\delta}=\cS-\frac{D}{M^{\,3/2}}\,\delta\,t^{-1/2}$,
$0\leqslant\delta<K$, and $\Huat(c_{t,\delta})$ defined in \eqref{erty5yeh}, we
have
\begin{equation*}
\begin{aligned}
\AInt{M}{u,c}(t)\,\big|_{\,u=\Huat(c_{t,\delta}),\,c=c_{t,\delta}}&\sim
1-\UGauss{0}{1}\left(-\delta+x_{\alpha}(\delta)\right)
\\[0pt]
&+\UGauss{0}{1}\left(-\delta-x_{\alpha}(\delta)\right)\exp\left\{\,2\,\delta
x_{\alpha}(\delta)\,\right\}=\alpha,\quad t\to\infty,
\end{aligned}
\end{equation*}
whence the result.
\end{proof}

\begin{theorem}\label{rthtrhjrthjXX}
For $c_{t,\delta}=\cS+\frac{D}{M^{\,3/2}}\,\delta\,t^{-1/2}$,
$0\leqslant\delta<K$, we have
\begin{equation}\label{xfgbfdnbfn}
\Huat(c_{t,\delta})=\frac{D}{M^{\,3/2}}\,x_{\alpha}(\delta)\sqrt{t}\big(1+\overline{o}(1)\big),\quad
t\to \infty,
\end{equation}
where $x_{\alpha}(\delta)$ is a solution to the equation
\begin{equation}\label{wqsrgtfh}
1-\UGauss{0}{1}\left(\delta+x\right)+\UGauss{0}{1}\left(\delta-x\right)
\exp\left\{-2\,\delta x\,\right\}=\alpha.
\end{equation}
\end{theorem}

\begin{proof}[Proof of Theorem~\ref{rthtrhjrthjXX}]
For $c_{t,\delta}=\cS+\frac{D}{M^{\,3/2}}\,\delta\,t^{-1/2}$,
$0\leqslant\delta<K$, and $\Huat(c_{t,\delta})$ defined in \eqref{xfgbfdnbfn},
we have
\begin{equation*}
\begin{aligned}
\AInt{M}{u,c}(t)\,\big|_{\,u=\Huat(c_{t,\delta}),\,c=c_{t,\delta}}&\sim
1-\UGauss{0}{1}\left(\delta+x_{\alpha}(\delta)\right)
\\[0pt]
&+\UGauss{0}{1}\left(\delta-x_{\alpha}(\delta)\right)\exp\left\{\,-2\,\delta
x_{\alpha}(\delta)\right\}=\alpha,\quad t\to\infty,
\end{aligned}
\end{equation*}
whence the result.
\end{proof}

It is noteworthy that in both Theorems~\ref{tyukytkXX} and \ref{rthtrhjrthjXX},
the expression $x_{\alpha}(0)$ is a solution to the equation
$1-\UGauss{0}{1}\left(x\right)=\alpha/2$, i.e., is equal to
$\NQuant{\alpha/2}$.

\begin{theorem}\label{tyurtutXX}
For\, $c>K\cS$, $K>1$, we have
\begin{equation}\label{aegwege}
\Huat(c)=\frac{D^{\,2}}{M^{\,2}}\,x_{\alpha}(c)\big(1+\overline{o}(1)\big),\quad
t\to\infty,
\end{equation}
where $x_{\alpha}(c)$ is a positive solution to the equation
\begin{equation}\label{sdfgbfdnbf}
\left(1-\UGauss{0}{1}\left(\frac{cM-2}{cM}x^{1/2}\,\right)\right)
\exp\left\{-2\,\frac{cM-1}{c^2M^{\,2}}\,x\,\right\}
+\UGauss{0}{1}\left(x^{1/2}\,\right)=1+\alpha.
\end{equation}
\end{theorem}

\begin{proof}[Proof of Theorem~\ref{tyurtutXX}]
For\, $c>K\cS$, $K>1$, and $u_{t}(c)=\frac{D^{\,2}}{M^{\,2}}\,z_{t}(c)$, where
$z_{t}(c)$ is a solution to \eqref{sdfgbfdnbf}, we have
$\AInt{M}{u,c}(t)\,\big|_{\,u=u_{t}(c)}\sim\alpha$, $t\to\infty$, whence the
result.
\end{proof}

The following result is an alternative (or addition) to
Theorem~\ref{ergtrewghwerg} for $0\leqslant c\leqslant\cS$, when both
$\Ruat(c)$ and $\Huat(c)$ tend to infinity, as $t\to\infty$. Of particular
interest is the right $t^{-1/2}$-neighborhood of the point $\cS$.

\begin{figure}[t]
\center{\includegraphics[scale=.8]{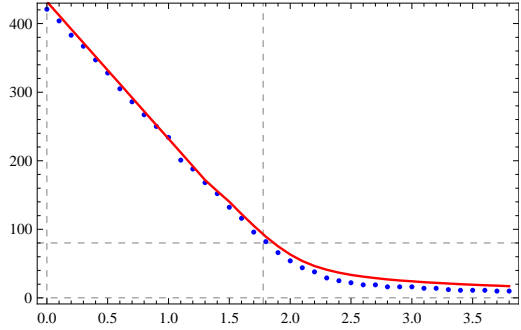}}
\caption{\small Graphs ($X$-axis ic $c$) of $\Huat(c)$ and simulated values of
$\Ruat(c)$ for $T$ exponentially distributed with parameter $\paramT=4/5$, $Y$
Pareto with parameters $\parA{Y}=10$, $\parB{Y}=0{.}05$, $\alpha=0{.}05$, and
$t=200$. Vertical grid line: $\cS=1{.}7778$. Horizontal grid line: simulated
$\Ruat(\cS)=80$.}\label{w4ytrujrtjrY}
\end{figure}

\begin{theorem}\label{asdgshbndfnf}
Suppose that conditions of Theorem~\ref{ergtrewghwerg} are satisfied. Then for
$0<c<\cS+\frac{D}{M^{\,3/2}}\,\delta\,t^{-1/2}$, $0\leqslant\delta<K$, we have
\begin{equation*}
\big|\,\Ruat(c)-\Huat(c)\,\big|=\overline{o}(1),\quad t\to\infty.
\end{equation*}
\end{theorem}

\begin{proof}[Proof of Theorem~\ref{asdgshbndfnf}]\label{asergffdf}
This is standard proof of the proximity of implicitly defined functions, when
they are defined by equations close to each other.
\end{proof}

A numerical illustration of the proximity of the original and heuristic
fixed-probability levels is Fig.~\ref{w4ytrujrtjrY}. Although $\Huat(c)$ seems
to be close to $\Ruat(c)$ for all $c\geqslant 0$, theoretically there is no
reason to expect that $\Ruat(c)$ and $\Huat(c)$ are close to each other for
$c>K\cS$, $K>1$, i.e., when they do not tend to infinity, as $t\to\infty$.

\subsection{Elementary asymptotic bounds for the fixed-probability level}\label{srtyhtrj}

\begin{figure}[t]
\center{\includegraphics[scale=.8]{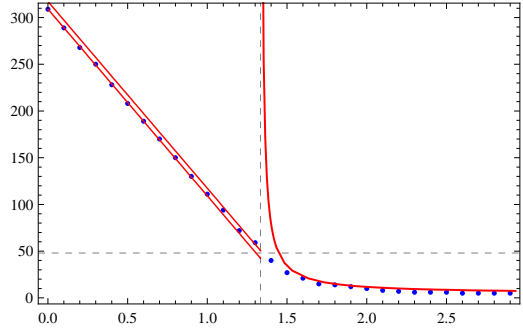}}
\caption{\small Graphs ($X$-axis is $c$) of two-sided bounds
\eqref{adsfgsdbsb}, when $0\leqslant c\leqslant\cS$, of the upper bound, when
$c>\cS$, and of simulated values of $\Ruat(c)$, drawn for $T$ which is Erlang
with parameters $\paramT=8/5$, $k=2$ and $Y$ exponentially distributed with
parameter $\paramY=3/5$, and $\alpha=0{.}05$, $t=200$. Vertical grid line:
$\cS=4/3$. Horizontal grid line: $\Ruat(\cS)=48$.}\label{egrrgd}
\end{figure}

For $0\leqslant c\leqslant\cS$, elementary asymptotic bounds
\begin{equation}\label{adsfgsdbsb}
\begin{aligned}
(\cS-c)\,t+\dfrac{D}{M^{\,3/2}}\,&\NQuant{\alpha}\sqrt{t}\,(1+{o}(1))\leqslant\Ruat(c)
\\[0pt]
&\leqslant(\cS-c)\,t+\dfrac{D}{M^{\,3/2}}\,\NQuant{\alpha/2}\sqrt{t}\,(1+{o}(1)),\quad
t\to\infty,
\end{aligned}
\end{equation}
follow straightforwardly from \eqref{asdfgrfjhf}.

For $c>\cS$, elementary upper bounds, quite satisfactory for $c>K\cS$ with
$K>1$ large enough, are also straightforward in many cases of interest. In
particular, when $T$ and $Y$ are exponentially distributed with parameters
$\paramT$ and $\paramY$, we have $\cS=\paramT/\paramY$ and (see, e.g.,
\cite{[Rolski=et=al.=1999]})
$\P\{\Tich{u,c}<\infty\}=(1-\adjustL/\paramY)\,e^{-\adjustL\,u}$ for all
$u\geqslant 0$, where $\adjustL=\paramY-\paramT/c$. This rewrites as
$\P\{\Tich{u,c}<\infty\}=\left(\paramT/(c\paramY)\right)
\,\exp\left\{-(\paramY-\paramT/c)\,u\right\}$; by simple calculations we have
\begin{equation*}
\Ruat(c)\leqslant\max\left\{\,0,-\frac{\ln\left(\alpha
c\paramY/\paramT\right)}{\paramY-\paramT/c}\right\},\quad c>\cS.
\end{equation*}

When $T$ is exponentially distributed with parameter $\paramT$ and the
distribution of $Y$ is light-tailed, but non-exponential, we have
$\cS=\paramT\,\E{Y}$ and (see, e.g., \cite{[Rolski=et=al.=1999]})
$\P\{\Tich{u,c}<\infty\}\leqslant e^{-\adjustL\,u}$ for all $u\geqslant 0$,
where $\adjustL$ is a positive solution to the equation
$\E\exp\{\adjustL\,Y\}=1+c\adjustL/\paramT$. Therefore, we have
\begin{equation*}
\Ruat(c)\leqslant-{\ln\alpha}/{\adjustL},\quad c>\cS.
\end{equation*}

When $Y$ is exponentially distributed with parameter $\paramY$ and the
distribution of $T$ is arbitrary, we have $\cS=1/(\paramY\,\E{T})$ and (see,
e.g., \cite{[Rolski=et=al.=1999]})
$\P\{\Tich{u,c}<\infty\}=(1-\adjustL/\paramY)\,e^{-\adjustL u}$ for all
$u\geqslant 0$, where $\adjustL$ is a positive solution to the equation
$\E\exp\{-\adjustL\,c\,T\}=1-\adjustL/\paramY$. Bearing in mind that
$1-\adjustL/\paramY\leqslant 1$, we have
\begin{equation*}
\Ruat(c)\leqslant-{\ln\alpha}/{\adjustL},\quad c>\cS.
\end{equation*}
In this case, the elementary bounds for the fixed-probability level are shown
in Fig.~\ref{egrrgd}.

\section{Derivatives of implicit function}

The derivatives of an implicit function defined by the equation $F(x,y)=0$,
$x,y\in\Rline$, can be obtained (see, e.g., \cite{[Widder=1947]}, Chapter~I,
\S~5.2 and \S~5.3) without finding this implicit function in closed form.

\begin{theorem}\label{etrtjt}
Assume that the function $F(x,y)$, $x,y\in\Rline$, possesses partial
derivatives up to second order, which are continuous in some neighborhood of a
solution $(x_0,y_0)$ of the equation $F(x,y)=0$. If $\frac{\partial}{\partial
y}F(x_0,y_0)\ne 0$, then there exists an $\epsilon>0$ and a unique continuously
differentiable function $f$ such that $f(x_0)=y_0$ and $F(x,f(x))=0$ for
$|x-x_0|<\epsilon$. Moreover, for $|x-x_0|<\epsilon$ we have
\begin{equation}\label{ewqrfghwre}
\left.f^{\,\prime}(x)=-\frac{\frac{\partial}{\partial
x}F(x,y)}{\frac{\partial}{\partial y}F(x,y)}\;\right|_{\,y=f(x)},
\end{equation}
and
\begin{equation}\label{ewryjhnmfm}
\begin{aligned}
f^{\,\prime\prime}(x)=&-\left(\frac{\frac{\partial^{2}}{\partial
x^{2}}F(x,y)}{\frac{\partial}{\partial
y}F(x,y)}-\frac{2\,\frac{\partial^{2}}{\partial x\,\partial y}F(x,y)\,
\frac{\partial}{\partial x}F(x,y)}{\left(\frac{\partial}{\partial
y}F(x,y)\right)^{2}}\left.+\frac{\frac{\partial^{2}}{\partial
y^{2}}F(x,y)\left(\frac{\partial}{\partial x}F(x,y)\right)^{2}}
{\left(\frac{\partial}{\partial
y}F(x,y)\right)^{3}}\right)\,\right|_{\,y=f(x)}.
\end{aligned}
\end{equation}
\end{theorem}


\begin{thebibliography}{99.}
\bibitem{[Beard-et-al.-1984]}
Beard,\,R.~E., Pentik{\"a}inen,\,T., and Pesonen,\,E. (1984) {Risk Theory. The
Stochastic Basis of Insurance.} 3-rd ed., Chapman and Hall, London, etc.

\bibitem{[Borovkov=1965]}
Borovkov,\,A.A. (1965) On the first passage time for one class of processes
with independent increments, {\TPA}, Vol.~10, 331--334.

\bibitem{[Borovkov=Dickson=2008]}
Borovkov,\,K.A., and Dickson,\,D.C.M. (2008) On the ruin time distribution for
a Spar\-re Andersen process with exponential claim sizes, {\IME}, Vol.~42,
1104--1108.

\bibitem{[Borovkov=2015]}
Borovkov,\,A.A. (2015) Integral theorems for the first passage time of an
arbitrary boundary by a compound renewal process,  Siberian Mathematical
Journal, Vol.~56, 5, 961--981.

\bibitem{[Chaudry=2002]}
Chaudry,\,M.A., and Zubair,\,S.M. (2002) Extended incomplete gamma functions
with applications, {\JMAA}, Vol.~274, 725--745.

\bibitem{[Daykin-et-al.-1996]}
Daykin,\,C.D., Pentik{\"a}inen,\,T., and Pesonen,\,M. (1996) {Practical Risk
Theory for Ac\-tu\-ari\-es.} Chapman and Hall, London, etc.

\bibitem{[Jorgensen=1982]}
J{\o}rgensen,\,B. (1982) {Statistical Properties of the Generalized Inverse
Gaussian Distribution.} Lecture Notes in Statistics. 9. New York, Berlin:
Springer.

\bibitem{[Good=1953]}
Good,\,I.J. (1953) The population frequencies of species and the estimation of
population parameters, {\BM}, Vol.~40, 237--260.

\bibitem{[Keilson=1963]}
Keilson,\,J. (1963) The first passage time density for homogeneous skip-free
walks on the continuum, {\AMS}, Vol.~34, 1003--1011.

\bibitem{[Kendall=1957]}
Kendall,\,D.G. (1957) Some problems in the theory of dams, {\JRSS}, Vol.~19,
207--212.

\bibitem{[Malinovskii=2009]}
Malinovskii,\,V.K. (2009) Scenario analysis for a multi-period diffusion model
of risk, {\ASB}, Vol.~39, 649--676.

\bibitem{[Malinovskii=2012]}
Malinovskii,\,V.K. (2012) Equitable solvent controls in a multi-period game
model of risk, {\IME}, Vol.~51, 599--616.

\bibitem{[Malinovskii=2014a]}
Malinovskii,\,V.K. (2014) Elementary bounds on the ruin capital in a diffusion
model of risk, {\Risks}, Vol.~2, 249--259; DOI information:
10.3390/risks2020249.

\bibitem{[Malinovskii=2014b]}
Malinovskii,\,V.K. (2014) Improved asymptotic upper bounds on ruin capital in
Lundberg model of risk, {\IME}, Vol.~55, 301--309.

\bibitem{[Malinovskii=2017a]}
Malinovskii,\,V.K. (2017) On the time of first level crossing and inverse
Gaussian distribution.\newline https://arxiv.org/pdf/1708.08665.pdf.

\bibitem{[Malinovskii=2017b]}
Malinovskii,\,V.K. (2017) Generalized inverse Gaussian distributions and the
time of first level crossing.\newline https://arxiv.org/pdf/1708.08671.pdf.

\bibitem{[Malinovskii=2018]}
Malinovskii,\,V.K. (2018) On approximations for the distribution of the time of
first level crossing.\newline https://arxiv.org/pdf/1803.09801.pdf.

\bibitem{[Malinovskii=2018=dan=1]}
Malinovskii,\,V.K. (2018) Approximations in the problem of level crossing by a
compound renewal process, {\DAN}, Vol.~483, No.~5, 622--625.

\bibitem{[Malinovskii=Kosova=2014]}
Malinovskii,\,V.K., and Kosova,\,K.O. (2014) Simulation analysis of ruin
capital in Sparre An\-der\-sen's model of risk, {\IME}, Vol.~59, 184--193.

\bibitem{[Malinovskii=Malinovskii=2017]}
Malinovskii,\,V.K., and Malinovskii,\,K.V. (2017) On approximations for the
distribution of first level crossing time.
https://arxiv.org/pdf/1708.08678.pdf.

\bibitem{[Morlat=1956]}
Morlat,\,G. (1956) Les lois de probabilit{\'e}s de Halphen, {\RSA}, Vol.~4,
No.~3, 21--46.

\bibitem{[Pentikainen-et-al.-1989]}
Pentik\"ainen,\,T., Bonsdorff,\,H., Pesonen,\,M., Rantala,\,J., and
Ruohonen,\,M. (1989) {Insurance Solvency and Financial Strength.} Finnish
Insurance Training and Pub\-li\-sh\-ing Co., Helsinki.

\bibitem{[Perreault=1999]}
Perreault,\,L., Bob{\'e}e,\,B., and Rasmussen,\,P.F. (1999) Halphen
Distribution System.~I: Mathematical and Statistical Properties, {\JHE}, Vol.~4
(3), 189--199.

\bibitem{[Rogozin=1966]}
Rogozin,\,B.A. (1966) Distribution of certain functionals related to boundary
value problems for processes with independent increments, {\TPA}, Vol.~11,
656--670.

\bibitem{[Rolski=et=al.=1999]}
Rolski,\,T., Schmidli,\,H., Schmidt,\,V., and Teugels,\,J. (1999) {Stochastic
Processes for Insurance and Finance.} John Wiley \& Sons, Chichester, etc.

\bibitem{[Sandstrom-2006]}
Sandstr{\"o}m,\,A. (2006) {Solvency. Models, Assessment and Regulation.}
Chapman \& Hall/CRC, Taylor \& Francis Group. Boca Raton, etc.

\bibitem{[Sandstrom=2011]}
Sandstr{\"o}m,\,A. (2011) {Handbook of Solvency for Actuaries and Risk
Managers: Theory and Practice.} Chapman \& Hall\,/\,CRC, Taylor \& Francis
Group. Boca Raton, etc.

\bibitem{[Seshadri=1997]}
Seshadri,\,V. (1997) Halphen's laws. In: Kotz,\,S., Read,\,C. B., Banks,\,D.L.
Encyclopedia of Statistical Sciences, Update, Vol.~1, 302--306. John Wiley \&
Sons, New York.

\bibitem{[Skorohod=1991]}
Skorohod,\,A.V. (1991) {Random Processes with Independent Increments.} Kluwer.

\bibitem{[Widder=1947]}
Widder,\,D.V. (1947) {Advanced Calculus.} Prentice-Hall, New York.

\bibitem{[Zolotarev=1964]}
Zolotarev,\,V.M. (1964) The first passage time of a level and the behaviour at
infinity of a class of processes with independent increments, {\TPA}, Vol.~9,
653--662.

\bibitem{[Zigangirov=1962]}
Zigangirov,\,K.S. (1962) Expression for the Wald distribution in terms of
normal distribution, {\RE} Vol.~7, 164--166.
\end{thebibliography}
\end{document}